\documentclass[a4paper,10pt]{article}

\usepackage[utf8]{inputenc}
\usepackage[english]{babel}

\usepackage{mathtools}
\usepackage{amsmath}
\numberwithin{equation}{section} 
\usepackage{amssymb}
\usepackage{amsthm}
\usepackage{cases}
\usepackage{stix}
\usepackage{bbm} 
\usepackage{pgf,tikz}
\usetikzlibrary{arrows}
\usepackage{xcolor}
\usepackage[symbol]{footmisc}
\renewcommand{\thefootnote}{\fnsymbol{footnote}}

\usepackage{textcomp}
\usepackage[pdftex,colorlinks]{hyperref}
\hypersetup{                    %
	      allcolors = blue}

\usepackage[textheight=9in, paperheight=11.25in, inner=3cm, outer=3cm]{geometry} 
\allowdisplaybreaks[2]

\allowdisplaybreaks
\newcommand{\R}{\mathbb{R}}
\newcommand{\N}{\mathbb{N}}

\newcommand{\dd}{\partial}

\DeclareMathOperator{\Tr}{Tr}                      
\renewcommand{\bar}[1]{\overline{#1}}

\renewcommand{\hat}{\widehat}

\newcommand{\al}{\alpha}
\newcommand{\be}{\beta}
\newcommand{\dl}{\delta}

\newcommand{\eps}{\varepsilon}

\newcommand{\ld}{\lambda}

\renewcommand{\O}{\Omega}

\definecolor{gr}{rgb}   {0.,   0.69,   0.23 }
\definecolor{bl}{rgb}   {0.,   0.5,   1. }
\definecolor{mg}{rgb}   {0.85,  0.,    0.85}
\definecolor{gy}{rgb}   {0.8,  0.8,   0.8}
\definecolor{yl}{rgb}   {0.8,  0.7,   0.}
\definecolor{or}{rgb}  {0.7,0.2,0.2}

\newcommand{\ip}[2]{\left\langle #1,#2\right\rangle}          

\newcommand{\1}{\hspace{0.5mm}\text{I}\hspace{0.5mm}}

\newcommand{\noi}{\noindent}

\newcommand\blfootnote[1]{
  \begingroup
  \renewcommand\thefootnote{}\footnote{#1}%
  \addtocounter{footnote}{-1}%
  \endgroup
 }

\renewcommand{\tilde}{\widetilde}

\theoremstyle{definition} \newtheorem{de}{Definition}[section]
                          
\theoremstyle{remark}     
\theoremstyle{plain}      \newtheorem{thm}[de]{Theorem}
		          \newtheorem*{thm*}{Theorem}
\theoremstyle{plain}      \newtheorem{cor}[de]{Corollary}
\theoremstyle{plain}      \newtheorem{pr}[de]{Proposition}
\theoremstyle{plain}      \newtheorem{lem}[de]{Lemma}
\theoremstyle{remark}     \newtheorem{rmk}[de]{Remark}
\theoremstyle{remark}     
		          \newtheorem*{rmk*}{Remark}
		          
\theoremstyle{plain}

\date{}
\begin{document}


\title{Linear Dynamics of the Semi-Geostrophic Equations\\in Eulerian Coordinates on $\mathbb{R}^{3}$}
\author{Stefania Lisai\footnote{\textit{Correspoding author:} {Heriot-Watt University and the Maxwell Institute for Mathematical Sciences, Edinburgh, UK}, \href{mailto:s.lisai@sms.ed.ac.uk}{s.lisai@sms.ed.ac.uk}.} \hspace{1mm}and
Mark Wilkinson\footnote{{Department of Mathematics, Nottingham Trent University, Nottingham, UK}, \href{mailto:mark.wilkinson@ntu.ac.uk}{mark.wilkinson@ntu.ac.uk}.
}
}
\maketitle

\begin{abstract}
\noindent We consider a class of steady solutions of the semi-geostrophic equations on $\R^3$ and derive the linearised dynamics around those solutions. The linear PDE which governs perturbations around those steady states is a transport equation featuring a pseudo-differential operator of order 0. We study well-posedness of this equation in $L^2(\R^3;\R^3)$ introducing a representation formula for the solutions, and extend the result to the space of tempered distributions on $\mathbb{R}^{3}$. We investigate stability of the steady solutions by looking at plane wave solutions of the linearised problem, and discuss differences in the case of the quasi-geostrophic equations.
\end{abstract}

\renewcommand{\thefootnote}{\arabic{footnote}}

\section{Introduction}

The aim of this work is to derive and analyse the linearised semi-geostrophic equation (LSG in what follows) associated to a class of quadratic steady solutions of the semi-geostrophic equations (referred to as SG in what follows) on the whole space $\mathbb{R}^{3}$. We present a well-posedness theory for this equation, and discuss the stability of these quadratic steady states. 
For a given quadratic steady state $\overline{P}$ of SG, the associated LSG in the conservative perturbation $\nabla\psi$ is the passive transport equation 
\begin{equation}\label{eq:LSG}
\frac{\partial}{\partial t}\nabla\psi+(\overline{u}\cdot\nabla)\nabla \psi+S^{T}\nabla\psi=\mathscr{F}^{-1}\circ\mathscr{M}_{m}\circ\mathscr{F}[\nabla\psi],
\end{equation}
where $\overline{u}$ is a $\overline{P}$-dependent Eulerian velocity field, $\mathscr{F}^{-1}\circ\mathscr{M}_{m}\circ\mathscr{F}$ is a pseudo-differential operator of order 0 with $\overline P$-dependent symbol $m$, $S$ is a $\overline P$-dependent matrix and $\mathscr F$ is the Fourier transform, defined in \eqref{eq:def_FT}. 
Although the well-posedness of the problem in $L^2(\R^3,\R^3)$ follows from the theory of strongly continuous semigroups on Banach spaces, we provide a representation formula for the solution, which allows us to work with the solutions explicitly. We present and extend this result to the space of tempered distributions $\mathcal S'(\R^3,\R^3)$, to allow for the treatment of global plane-wave solutions, which are not integrable functions on the whole space $\R^3$.
We also introduce concepts of stability for this family of steady solutions of the semi-geostrophic equations, by studying the long-term behaviour of wave-like solutions, in analogy with the results of \cite{CC86:NS_planewaves} on the Navier--Stokes equations.  Finally, we briefly mention how the same methods can be applied to the quasi-geostrophic equations, and discuss similarities and differences between the two theories.

\subsection{Contributions of this Paper}

The main novelty of this paper is the introduction of an analytical framework in which one can analyse the linearised semi-geostrophic equations for steady solutions that are globally defined in space. We consider quadratic steady solutions, and this choice simplifies the derivation, as we will see in Section \ref{sec:LSG_deriv}, and it allows us to write the linearised equation in a very explicit manner: see \eqref{eq:LSG}. Our choice of steady solutions allows us to prove existence of solutions to LSG, and study the stability of the steady solutions, in a sense that we specify in Definition \ref{de:LSG_stability}. It is important to notice that the derivation of LSG is only formal here: a rigorous derivation would require a well-posedness theory for the semi-geostrophic equations on the whole space $\R^3$, which is still an open problem, to the best of the authors' knowledge.

The authors believe that the framework presented in this paper can be a first step towards the understanding the long-time behaviour of more general steady solutions of the semi-geostrophic equations, and offer an interesting first comparison between the behaviour of common steady solutions of the semi-geostrophic and the quasi-geostrophic model. In fact, in Section \ref{sec:LSG_LQG}, we discuss how the stability analysis expects different behaviour for some of the common steady solutions of the semi-geostrophic and the quasi-geostrophic equations.

To the best of the authors' knowledge, this is the first analysis of the stability of globally-defined steady solutions of the semi-geostrophic equations, although in the physics literature \cite{Ren98} analyses the linear stability of parallel basic flow in the baroclinic semi-geostrophic equations in geometric coordinates.
 
\subsection{The Semi-Geostrophic Equations in Eulerian Coordinates}

The semi-geostrophic equations can be written in various coordinate systems, depending on the analytical tools that we wish to employ (e.g. optimal transport theory) and the aim of our research (e.g. numerical simulation, well-posedness theory). The first formulation of SG in Eulerian coordinates was introduced in the 1950s in \cite{Eliassen48} and studied by Hoskins in the 1970s. However, the model attracted the interest of the mathematical community only in the late 1990s when Benamou and Brenier introduced in \cite{BB98} a new formulation, in a set of coordinates which are now known as \emph{geostrophic} or \emph{dual} coordinates, and showed that tools from the theory of optimal transport could be applied to study the well-posedness of the problem. Since then, the problem is typically studied in geostrophic coordinates, because of its connection to optimal transport: see, among the others, \cite{ACDF12:periodic} and \cite{ACDF12:convex}. We invite the reader to see \cite{LW19} for a review of the main formulations and results in the semi-geostrophic theory. 

For our purposes, we shall consider SG in Eulerian coordinates, which is the original setting in which the equations were derived and studied by Hoskins \cite{Hoskins75} in 1975. In Eulerian coordinates, SG takes the form of an active transport equation given by
\begin{equation}\label{eq:SG}
\left\{\begin{array}{l}
\displaystyle \frac{\partial}{\partial t} \nabla P + (\mathscr U[\nabla P]\cdot \nabla )\nabla P = J(\nabla P - \text{id}_{\O}),\vspace{2mm}\\
 \nabla P(\cdot,0)= \nabla P_0,
 \end{array}
 \right.
\end{equation}
in the unknown time-dependent conservative vector field $\nabla P:\Omega\times \R\to\R^3$, where $\Omega\subseteq\mathbb{R}^{3}$. In the above, the matrix $J\in\R^{3\times3}$ is given by
\begin{equation}\label{eq:def_J}
    J:= \begin{pmatrix}
            0&-1&0\\1&0&0\\0&0&0
        \end{pmatrix},
\end{equation}
the function $\text{id}_{\O}:x\mapsto x$ denotes the identity function on $\O$.
The operator $\mathscr U:\nabla P\mapsto u$ is the formal solution operator associated to the  boundary-value problem for the div-curl system on $\O$ given by
\begin{equation}\label{eq:divcurl}
    \left\{
        \begin{array}{l}
            \nabla \wedge (D^2 P \,u) = \nabla\wedge
            J(\nabla P-\text{id}_{\R^3}), \vspace{2mm}\\
            \nabla \cdot u=0.
        \end{array}
    \right.
\end{equation}
This boundary value problem is endowed with the usual no-slip condition $u\cdot n=0$ on $\partial\O$ if $\O$ is a bounded domain in $\R^3$. To guarantee well-posedness of $\mathscr{U}$ in the case when $\Omega=\mathbb{R}^{3}$, further conditions must be imposed on $u$. In this article, we are interested in the linear dynamics of SG around a special family of steady solutions in the whole space case $\Omega=\mathbb{R}^{3}$. This family of steady solutions is realised by globally-supported quadratic strictly convex functions: for a given symmetric positive definite matrix $A\in \text{Sym}^+_3(\R)\subset \R^{3\times3}$, we consider those steady solutions $\overline{P}$ of \eqref{eq:SG} of the form
\begin{equation}\label{eq:steadysol}
    \overline P(x):= \frac 12 x\cdot Ax, \qquad \forall x\in\R^3.
\end{equation}
Indeed, one can show that $\nabla \overline P$ is a steady solution of \eqref{eq:SG} with corresponding velocity field $\overline{u}$ given by
\begin{equation}\label{eq:steadyvel}
    \overline u(x):=Sx,
\end{equation}
where the matrix $S$ is given by $S:=A^{-1}J(A-I)$ and $I\in\R^{3\times3}$ is the identity matrix.

\subsection{Cullen's Stability Principle}

The choice of steady solutions \eqref{eq:steadysol} allows us to derive the linearised semi-geostrophic equation by formally taking Fr\'echet derivatives of \eqref{eq:SG} and \eqref{eq:divcurl} at $\nabla\overline P$ in the direction of a smooth conservative vector field $\nabla \psi$ (see Section \ref{sec:LSG_deriv}). Using the fact that the strictly-convex $\bar P$ has constant Hessian, we will show that LSG can be written as the abstract Cauchy problem
\begin{equation}\label{eq:LSG_abs}
    \left\{\begin{array}{l}
        \frac{d}{d t}\phi = \mathscr L\phi,\vspace{2mm}\\
        \phi(0)=\phi_0,
    \end{array}\right.
    \end{equation}
where the linear operator $\mathscr L$ is given by
\begin{equation}\label{eq:LSG_L}
    \mathscr L \phi:= -(\overline u \cdot \nabla)\phi - S^T\phi + \mathscr F^{-1}\circ\mathscr M_m\circ \mathscr F[\phi],
\end{equation}
as in \eqref{eq:LSG}.
In the above, the operator $\mathscr M_m$ in \eqref{eq:LSG_abs} is the multiplication operator $\mathscr M_m[\phi]=m\phi$ with symbol $m:\R^3\setminus\{0\}\to\R$ which depends on the choice of the matrix $A$ and is given by
\begin{equation}\label{eq:def_m}
 m(x):= 2\frac{x\cdot A^{-1}Jx}{x\cdot A^{-1}x},
\end{equation}
for any $x\in\R^3\setminus\{0\}$. Observe that $m\in L^{\infty}(\R^3)\cap \mathcal C^{\infty}(\R^3\setminus\{0\})$.

The primary reason we work with this specific family of steady solutions is due to the well-known Stability Principle of Cullen, which is adopted by most authors when analysing SG. Indeed, Cullen's stability principle (introduced for the first time in \cite{CS87}) postulates that physically stable solutions of SG are those for which $P(\cdot,t)$ is a convex function in its spatial variable for any $t$. For more about the stability principle, we invite the reader to see \cite[Ch.~3]{Cullen:book} and \cite{CKPW19}. 

As an interesting consequence of our analysis in this paper, we show that whilst steady solutions of the shape \eqref{eq:steadysol} are stable in the sense of Cullen, they are in fact \emph{dynamically unstable} in many natural topologies. An additional structural benefit to considering steady states of the form \eqref{eq:steadysol}, we show in Section \ref{sec:LSG_deriv}, is that all those steady solutions with constant Hessian $D^2\bar P$ result in a linearised equation that involves a pseudo-differential operator of order 0 on $\mathbb{R}^{3}$. We do not investigate the linearisation of SG around other steady solutions in this paper, and we leave it for future work. Let us also mention that the corresponding problem in bounded domains $\Omega\subset\mathbb{R}^{3}$ is substantially harder, with the existence of non-trivial steady solutions of \eqref{eq:SG} on bounded domains being an open problem.

\subsection{Characterisation of Steady Quadratic Flows}

We observed that the velocity field induced by a geopotential of the form \eqref{eq:steadysol} is given by the linear transformation in \eqref{eq:steadyvel}. In particular, different choices for the positive-definite symmetric matrix $A$ will produce different steady Eulerian flow fields. To show this, we consider a general symmetric matrix 
\begin{equation}\label{eq:LSG_def_A}
 A:=\begin{pmatrix}
     a&b&c\\b&d&e\\c&e&f
    \end{pmatrix},
\end{equation}
with the coefficients $a,b,c,d,e,f\in\R$ chosen in such a way that $A$ is positive definite. Then, the flow velocity is given by $\overline u(x)=Sx$, with $S=A^{-1}J(A-I)$. In particular, $S$ can be written in terms of the coefficients of $A$ 
and it is easily verified that its spectrum is $\sigma(S)=\{0,\lambda_A,-\lambda_A\}$, where 
\begin{equation}\label{eq:def_mu}
 \lambda_A:= \sqrt{\frac{\mu_A}{\det A}}, \qquad \mu_A:=af-c^2+df-e^2-f -\underbrace{(adf-ae^2-b^2f+2bce-c^2d)}_{=\det A}.
\end{equation}
Hence, the matrix $S$ has real eigenvalues if and only if $\mu_A\geq0$, and it has two purely imaginary eigenvalues otherwise. We define the two sets in the space of positive definite symmetric real matrices
\begin{align}\label{eq:flowsets}
\begin{split}
 \mathscr A_+:= \left\{ A\in\text{Sym}_+(3,\R)\, :\, \mu_A>0 \right\},\\
 \mathscr A_-:= \left\{ A\in\text{Sym}_+(3,\R)\, :\, \mu_A<0 \right\}.
 \end{split}
\end{align}
As a result, the set $\mathscr A_+$ corresponds to \emph{hyperbolic flows}, whereas the set $\mathscr A_-$ corresponds to \emph{elliptic flows}. The set $\mathscr A_0:=\{A\in\text{Sym}_+(3,\R)\, :\, \mu_A=0 \}$ contains those matrices for which $\bar u\equiv 0$, and moreover this case is not mathematically interesting as the corresponding Fourier multiplier $m$ is null. It is readily seen that the set $\mathscr A_0$ is non-empty as it contains the identity matrix $I$.

\subsection{Statement of Main Results}

For reasons pertaining to plane wave perturbations that we outline in the sequel, we aim to construct solutions of \eqref{eq:LSG} in $\mathcal S'(\R^3,\R^3)$ by duality, thereby considering the \emph{adjoint} LSG given by the abstract Cauchy problem
\begin{equation}\label{eq:ALSG}
    \left\{ \begin{array}{l}
                \frac{d}{dt} \psi = \mathscr K\psi,\vspace{2mm}\\
                \psi(0) = \psi_0,
            \end{array}\right.
\end{equation}
where the operator $\mathscr K$ is the adjoint of $\mathscr L$, namely
\begin{equation}\label{eq:LSG_K}
    \mathscr K\psi:= (\overline u \cdot \nabla)\psi - S\psi + \mathscr F\circ\mathscr M_m\circ \mathscr F^{-1}[\psi].
\end{equation}
The reader will observe that symbol $m$ defined in \eqref{eq:def_m} is even, therefore a change of variable shows that $\mathscr F\circ\mathscr M_m\circ \mathscr F^{-1}= \mathscr F^{-1}\circ\mathscr M_m\circ \mathscr F$. This means that the analysis of the operator $\mathscr K$ in \eqref{eq:LSG_K} is essentially the same as that of the operator $\mathscr L$ defined in \eqref{eq:LSG_L}, up to a change of sign in the first term and transposing the matrix $S$ in the second. In particular, all of the results that we prove for $\mathscr L$ hold for $\mathscr K$ mutatis mutandis. 

\begin{rmk}
Since the velocity $\overline u$ corresponding the steady solution \eqref{eq:steadysol} is given by $\overline u(x)=Sx$, the associated flow map $F_{\overline u}(x,t)= e^{tS} x$ is a volume-preserving diffeomorphism of $\R^3$ for any time $t$, and we expect LSG in Eulerian coordinates to be equivalent to the formulation in Lagrangian coordinates. We will not comment further on the Lagrangian formulation, except for mentioning that existence of strong solutions  of LSG in Lagrangian coordinates can be proved in $L^2(\R^3,\R^3)$ by applying the theory of \emph{evolution systems} generated by time-dependent operators, as presented in \cite[Ch.~5]{Pazy12}. In fact, LSG in Lagrangian coordinates (up to a rescaling) is given by 
\begin{equation}\label{eq:LSG_Lagrangian}
    \left\{ \begin{array}{l}
                \frac{d}{dt} \Phi (t) = \mathcal A(t)\Phi(t) 
                ,\vspace{2mm}\\
                \Phi (0) = \Phi_0,
            \end{array}
    \right.
\end{equation}
where the time-depending operator $\mathcal A(t)$ is defined as 
\begin{equation*}
    \mathcal A(t):= \mathscr F^{-1}\circ \mathscr M_{\tilde m(t)}\circ \mathscr F,  
 \end{equation*}
 with $\tilde m$ defined in \eqref{eq:LSG_def_symbols}.
\end{rmk}
As we mentioned above, our purpose is firstly to present the $L^2$ existence theory of LSG and extend it to tempered distributions. The following is our first main result:

\begin{thm}\label{thm:LSG_existence_S'}
     Given an initial tempered distribution $\eta_0\in\mathcal S'(\R^3,\R^3)$, there exists a strong solution in $\mathcal S'(\R^3,\R^3)$ of \eqref{eq:LSG_abs}, in the sense of Definition \ref{de:LSG_solutions} (iv).
\end{thm}

The reason we extend the existence theory of LSG to allow for less regular solutions than those which take values in $L^2$ is so that our solution theory admits plane-wave solutions of the form
\begin{equation}\label{eq:LSG_planewavesol}
    \phi(x,t) = \nabla f(x,t), \qquad f(x,t)=a(t)e^{2\pi ik(t)\cdot x},
\end{equation}
where $a:\R\to\R$ and $k:\R\to\R^3$ are appropriately chosen functions. These solutions are not integrable on $\mathbb{R}^{3}$, with $\phi(\cdot,t)\notin L^p(\R^3,\R^3)$ for any $p\in [1,\infty)$. Nevertheless, we consider the corresponding regular distribution $\eta(t)=\eta_{\phi(\cdot,t)}$ and characterise the functions $a$ and $k$ for which $\eta$ is a strong solution of LSG in $\mathcal S'$. Then, for these specific solutions, we look at their stability in the following sense.

\begin{de}\label{de:LSG_stability}
    We say that a steady solution $\bar P$ of \eqref{eq:SG} is \emph{stable to plane-wave perturbations} if any plane-wave solution $\phi$ of the form \eqref{eq:LSG_planewavesol} of the associated LSG belongs to $L^{\infty}(\R^3\times\R,\R^3)$. Otherwise, we say that $\bar P$ is \emph{unstable to plane-wave perturbations}, i.e. there exists a plane-wave solution $\phi$ of LSG of the form \eqref{eq:LSG_planewavesol} such that $\|\phi(\cdot,t)\|_{L^{\infty}(\R^3,\R^3)}$ is unbounded.
\end{de}
With this in place, let us state the main stability result of this work.
\begin{thm}\label{thm:LSG_stability}
    The steady solution $\bar P$ defined in \eqref{eq:steadysol} associated to a matrix $A$ is stable to plane-wave perturbations if it corresponds to an elliptic flow, i.e. $A\in\mathscr A_-$ as in \eqref{eq:flowsets}, and it is unstable to plane-wave perturbations if it corresponds to an hyperbolic flow, i.e. $A\in\mathscr A_+$.
\end{thm}

\subsection{Notation and Main Definitions}

For the reader's convenience, we fix here the notation that will be used throughout. The space $\mathcal S(\R^3,\R^3)$ is the space of Schwartz functions, i.e. the set of all $f\in \mathcal C^{\infty}(\R^3,\R^3)$ such that 
\begin{equation}\label{eq:S_seminorms}
 \|f\|_{\al,\be}:=\sup_{x\in\R^3}|x^{\al}D^{\be}f(x)|<\infty,
\end{equation}
for any $\al,\be\in\N^3$ multi-indices, where $x^{\al}:=x_1^{\al_1}x_2^{\al_2}x_3^{\al_3}$ and $D^{\be} = \frac{\partial^{|\be|}}{\partial x_1^{\be_1}\partial x_2^{\be_2}\partial x_3^{\be_3}}$. The set $\mathcal S(\R^3,\R^3)$ endowed with the weak topology generated by the family of seminorms $\{\|\cdot\|_{\al,\be}\}_{\al,\be\in\N^3}$ is a Fr\'echet space.  One can prove that the topology generated by the family of seminorms is equivalent to the topology generated by the family of their finite linear combinations.
The space $\mathcal S'(\R^3,\R^3)$ is the space of tempered distributions endowed with the weak*-topology.
We use the notation $\mathcal S$ and $\mathcal S'$ in place of $\mathcal S(\R^3,\R^3)$ and $\mathcal S'(\R^3,\R^3)$ respectively, unless it is necessary to specify a different range of the Schwartz functions.
$\mathscr F$ and $\mathscr F^{-1}$ denote the Fourier transform and its inverse respectively. To avoid confusion due to the notation, we specify that we define the Fourier transform and its inverse on $L^1(\R^3,\R^3)$ as follows:
\begin{equation}\label{eq:def_FT}
    \mathscr F[\phi](\xi):=\int_{\R^3}\phi(x)e^{-2\pi i\xi\cdot x}\,dx, \qquad 
    \mathscr F^{-1}[\phi](x):=\int_{\R^3}\phi(\xi)e^{2\pi i\xi\cdot x}\,d\xi,
\end{equation}
and extend this definition to $L^2(\R^3,\R^3)$ by density in the usual way.

In order to construct solutions to LSG, we will need to consider the following Fourier symbols, defined from $m$ in \eqref{eq:def_m}:
\begin{align}\label{eq:LSG_def_symbols}
    \tilde m(t;\xi):=m(e^{-tS^T}\xi),\qquad \tilde M(t;\xi):=e^{\int_0^t \tilde m(r;\xi)\,dr},\\
    \bar m(t;\xi):=m(e^{tS^T}\xi),\qquad \bar M(t;\xi):=e^{\int_0^t \bar m(r;\xi)\,dr}.
\end{align}
The symbol $\tilde m(t;\cdot)$ is even, therefore $\mathscr F^{-1}\circ \mathscr M_{\tilde m(t;\cdot)}\circ \mathscr F=\mathscr F\circ \mathscr M_{\tilde m(t;\cdot)}\circ \mathscr F^{-1}$, and the same holds for $\bar m(t;\cdot), \tilde M(t;\cdot)$ and $\bar M(t;\cdot)$.

We now introduce the notions of solutions that will be of our interest in what follows. 
\begin{de}\label{de:LSG_solutions}
    \begin{itemize}
        \item [(i)] We say that a function $\phi:\R\to L^2(\R^3,\R^3)$ is a \emph{strong solution} in $L^2(\R^3,\R^3)$ of LSG \eqref{eq:LSG_abs} if $\phi\in \mathcal C^1(\R,L^2(\R^3,\R^3))$ and it satisfies \eqref{eq:LSG_abs} in $L^2(\R^3,\R^3)$ for all $t\in\R$.
        
        \item [(ii)] We say that a function $\phi$ is a \emph{weak solution} in $L^2(\R^3,\R^3)$ of LSG \eqref{eq:LSG_abs} if $\phi\in \mathcal C^0 (\R;L^2(\R^3;\R^3))$, it satisfies the following integral equation
        \begin{align*}
            \int_{\R}\int_{\R^3} \phi(x,t)\,\cdot\, &\bigg[\partial_t f(x,t)+ (\overline u(x) \cdot \nabla)f(x,t) - Sf(x,t) + \mathscr F\circ \mathscr M_m \circ \mathscr F^{-1}[f(\cdot,t)](x)\bigg]\,dx\,dt=0
        \end{align*}
        for any test function $f\in\mathcal C^{\infty}_C(\R^3\times\R;\R^3)$, and $\phi(\cdot,0)=\phi_0$ in $L^2(\R^3,\R^3)$.

        \item [(iii)] We say that a function $\phi:\R\to\mathcal S$ is a \emph{strong solution} in $\mathcal S$ of LSG \eqref{eq:LSG_abs} if $\phi\in \mathcal C^1(\R,\mathcal S)$ and it satisfies the differential equation and the initial condition in \eqref{eq:LSG_abs} in $\mathcal S$ for all $t\in\R$.
        
        \item [(iv)] We say that $\eta:\R\to\mathcal S'$ is a \emph{strong solution} in $\mathcal S'$ of LSG \eqref{eq:LSG_abs} if $\eta\in \mathcal C^1(\R,\mathcal S')$ and it satisfies the differential equation and the initial condition in \eqref{eq:LSG} in $\mathcal S'$ for all $t\in\R$, i.e.
        \begin{align*}
            \ip{\frac d{dt} \eta(t)}{\psi} = \ip{\mathscr L \eta(t)}{\psi}\quad\text{and}\quad \ip{\eta(0)}{\psi}=\ip{\eta_0}{\psi},\qquad \psi\in\mathcal S,\forall t\in\R.
        \end{align*}
        The time derivative in the expression above is to be interpreted in the weak*-sense in $\mathcal S'$.
    \end{itemize}    
\end{de}

\noi
The operator $\mathscr L$ defined in \eqref{eq:LSG_L} consists of three additive terms: we will need to refer to them and the groups of operators that they generate, therefore we list here their definitions for the reader's convenience.

\begin{de}\label{de:LSG_semigroups}
    \begin{itemize}
    \item [(i)] We consider the three additive terms that constitute the operator $\mathscr L=L_1+L_2+L_3$ defined in \eqref{eq:LSG_L}:
    \begin{align*}
        &L_1[\phi]:=-(\bar u\cdot \nabla)\phi,\\
        &L_2[\phi]:=-S^T\phi,\\
        &L_3[\phi]:=\mathscr F^{-1}\circ\mathscr M_m\circ\mathscr F[\phi]=\mathscr F^{-1}[m\hat{\phi}].
    \end{align*}
    \item [(ii)] We define the families of operators $\{T_1(t)\}_{t\in\R}$, $\{T_2(t)\}_{t\in\R}$ and $\{T_3(t)\}_{t\in\R}$ as follows: 
        \begin{align*}
         &T_1(t)[\phi](x)=\phi(e^{-tS}x),\\
         &T_2(t)[\phi](x)=e^{-tS^T}\phi(x),\\
         &T_3(t)[\phi](x)=\mathscr F^{-1}\circ\mathscr M_{\tilde M(t)}\circ\mathscr F[\phi] = \mathscr F^{-1}\big[\tilde M(t;\cdot)\hat{\phi}\big].
        \end{align*}
    \item [(iii)] For any $t\in\R$, we define the operator $G(t)$ as $G(t)[\phi]=T_1(t)T_2(t)T_3(t)[\phi]$, i.e.
    \begin{equation}\label{eq:LSG_rep_form}
        G(t)[\phi](x) = \int_{\R^3} e^{-tS^T}\hat {\phi} (\xi) \tilde M(t;\xi) e^{2\pi i \xi\cdot e^{-tS}x}\,d\xi.
    \end{equation}
    \end{itemize}
  
  With an abuse of notation, we use the same symbol to denote an operator defined in different topologies. For instance, we denote by $\mathscr L$ both the operator in \eqref{eq:LSG_L} defined on the space of distributions $\mathcal S'$ and its ``restriction'' to functions in $L^2(\R^3,\R^3)$ and $\mathcal S$. The topology that we work with will be clear by the context. By this, we mean that if $\eta_f$ is a regular distribution corresponding to the function $f\in L^2(\R^3,\R^3)$, then $\mathscr L \eta_f = \eta_{\mathscr L f}$. For the sake of clarity, we mention that by \emph{regular distribution} we mean a distribution $\eta\in\mathcal S'$ such that there exists a function $f\in L^p(\R^3,\R^3)$ for some $p\in[1,\infty]$, such that 
    \begin{align*}
        \ip{\eta}{\psi} = \int_{\R^3} f(x)\cdot \psi(x) \,dx, \qquad \forall \psi\in\mathcal S.
    \end{align*}

\end{de}

Some concepts of the theory of semigroups will be employed throughout the paper: the reader can observe that $L_1$ is the infinitesimal generator of the strongly continuous group $\{T_1(t)\}_{t\in\R}$ in $L^2(\R^3,\R^3)$ The theory of semigroups on Banach spaces is classical (and can be found in \cite{Pazy12}), whilst the theory of semigroups on locally convex topological spaces can be found in \cite[Ch.~IX]{Yosida} and in \cite{AK:Trotter}.

\subsection{Structure of the Paper}
We present the proofs of the theorems above in the following sections, proceeding as follows:
in Section \ref{sec:LSG_deriv}, we present the formal derivation of LSG for a conservative vector field, starting from SG as written in \eqref{eq:SG}, and we comment on the degeneracy of the symbol $m$, distinguishing between the matrices $A$ that give rise to a trivial pseudo-differential operator and those that do not. 
In Section \ref{sec:LSG_existence}, we start by presenting the existence theory for solutions in $L^2(\R^3,\R^3)$, and then provide the existence result in the space of Schwartz functions $\mathcal S$. We conclude the section with the proof of Theorem \ref{thm:LSG_existence_S'} and discuss the regularity of solutions with regular initial datum, connecting the statement with the existence in $L^2(\R^3,\R^3)$.
In Section \ref{sec:LSG_stability}, we focus on plane-wave solutions of LSG and look at their long-term behaviour in the case of both elliptic and hyperbolic flows, proving Theorem \ref{thm:LSG_stability}. In Section \ref{sec:LSG_LQG}, we briefly investigate the stability of the same family of steady solutions in the quasi-geostrophic theory, and draw a comparison of the results in the two cases.

\section{Derivation of LSG}\label{sec:LSG_deriv}

We provide a formal derivation of LSG in Eulerian coordinates, as in \eqref{eq:LSG}, from SG in the form  presented in \eqref{eq:SG}. 
We observe that SG can be written as 
\begin{equation*}
 \frac{\dd}{\dd t}\nabla P = \mathscr N[\nabla P],
\end{equation*}
where the nonlinear operator $\mathscr N$ is defined formally as 
\begin{equation}\label{eq:def_N}
\mathscr N[\nabla P]:=-(D^2 P)^T\mathscr U[\nabla P] + J(\nabla P-\text{id}_{\R^3}), 
\end{equation}
so the linearised SG at $\nabla \overline P$ is obtained by formally taking the Fr\'echet derivative of the operator $\mathscr N$ at $\nabla\overline P$ and it is of the form
\begin{equation}\label{eq:linearisedACP}
 \frac{\dd}{\dd t}\nabla \psi = d\mathscr N[\nabla \overline P; \nabla \psi],
\end{equation}
in the unknown conservative perturbation $\nabla\psi\in\mathcal C^{\infty}(\R,\mathcal S)$.


\begin{pr}
    Given the smooth steady solution $\nabla \overline P$ of SG defined in \eqref{eq:steadysol}, the formal Fr\'echet derivative of the operator $\mathscr N$ defined in \eqref{eq:def_N} at $\nabla \overline P$ acting on a conservative vector field $\nabla\psi$ can be written as the operator $\mathscr L$ defined in \eqref{eq:LSG_L}.
\end{pr}

\begin{proof}
    Formally, the Fr\'echet derivative of the operator $\mathscr N$ at $\nabla \overline P$ in the direction $\nabla \psi$ is given by 
    \begin{equation}\label{eq:def_dN}
        d\mathscr N[\nabla \overline P;\nabla \psi] = - A\ d\mathscr U[\nabla\overline P;\nabla \psi] - D^2\psi\, \overline u + J\nabla \psi,
    \end{equation}
    where the linear operator $d\mathscr U[\nabla\overline P;\cdot]$ is the Fr\'echet derivative of the nonlinear operator $\mathscr U$ at the steady solution $\nabla \overline P$ in the direction $\nabla\psi$. One can show that $v:=d\mathscr U[\nabla\overline P;\nabla\psi]$ coincides with a solution of the div-curl system
    \begin{equation}\label{eq:d_divcurl}
        \left\{ 
        \begin{array}{l}
            \nabla\wedge (A\ v) = - \nabla \wedge( D^2\psi\, \overline u) + \nabla\wedge J\nabla \psi,\vspace{2mm}\\
            \nabla\cdot v=0.
        \end{array}
        \right.
    \end{equation}
We consider $d\mathscr N[\nabla\overline{P};\nabla \psi]$, defined in \eqref{eq:def_dN}, and observe in \eqref{eq:d_divcurl} that it is a curl-free quantity, therefore there exists a scalar function $E$ such that 
    \begin{equation*}
        \nabla E = -A\ d\mathscr U[\nabla\overline P;\nabla \psi] - D^2\psi\, \overline u + J\nabla \psi.
    \end{equation*}
    As the matrix $A$ is positive definite, and $d\mathscr U[\nabla\overline P;\nabla \psi]$ is divergence-free, we can write an elliptic equation for $E$ in divergence form
    \begin{equation*}
        \nabla\cdot (A^{-1}\nabla E) = F_1+F_2:={-\nabla\cdot (A^{-1} D^2\psi\,\overline u)} + {\nabla\cdot(A^{-1}J\nabla\psi)}
    \end{equation*}
    By formally taking Fourier transform, we have that 
    \begin{equation*}
        \hat E= \hat E_1+\hat E_2, \qquad \hat{E}_i(\xi):= -\frac1{4\pi^2\, \xi\cdot A^{-1}\xi} \hat {F_i} (\xi).
    \end{equation*}
    We start looking at $E_1$: by considering the Fourier transform of $F_1$
    \begin{align*}
        \hat{F_1}(\xi)
        &= -\int_{\R^3} \nabla\cdot \left(A^{-1}D^2\psi(x)Sx\right)e^{-2\pi i\xi\cdot x}\,dx\\
        &=2\pi i SA^{-1}\xi\cdot \hat{\nabla \psi}(\xi)\, +\, 4\pi^2  (\xi\cdot A^{-1}\xi)  \mathscr F[\overline u\cdot \nabla \psi](\xi),
    \end{align*}
    one can write the explicit formula for $E_1$:
    \begin{align}\label{eq:nablaE1}
        E_1(x)
        &= -\frac1{4\pi^2} \int_{\R^3} \frac{\hat{F_1}(\xi)}{\xi\cdot A^{-1}\xi}\, e^{2\pi i\xi\cdot x} \,d\xi\\\notag
        &=\frac 1{2\pi i} \int_{\R^3} \frac{SA^{-1}\xi\cdot \hat{\nabla\psi}(\xi)}{\xi\cdot A^{-1}\xi}\,e^{2\pi i\xi\cdot x}\,d\xi - \overline u(x)\cdot \nabla \psi(x)\\  \notag
        &=\int_{\R^3} \frac{\xi\cdot SA^{-1}\xi}{\xi\cdot A^{-1}\xi}\hat{\nabla\psi}(\xi)\,e^{2\pi i \xi\cdot x}\,d\xi - S^T\nabla\psi(x) - \left(Sx\cdot \nabla\right)\nabla\psi(x).
    \end{align}
    Similarly, to find $ E_2$, we first write
    \begin{align*}
        \hat{F_2}(\xi) 
        &= \int_{\R^3} \nabla\cdot \left(A^{-1}J\nabla\psi(x)\right)\,e^{-2\pi i \xi\cdot x}\,dx\\
        &=-2\pi i JA^{-1}\xi\cdot \hat{\nabla\psi}(\xi), 
    \end{align*}
    and compute the inverse Fourier transform to find the explicit formula for $E_2$
    \begin{align*}
        E_2(x)
        &= -\frac 1{2\pi i}\int_{\R^3} \frac{JA^{-1}\xi\cdot \hat{\nabla\psi}(\xi)}{\xi\cdot A^{-1}\xi}\,e^{2\pi i \xi\cdot x}\,d\xi\\
        &=-\int_{\R^3} \frac{\xi\cdot JA^{-1}\xi}{\xi\cdot A^{-1}\xi}\hat{\nabla\psi}(\xi)\,e^{2\pi i \xi\cdot x}\,d\xi.
    \end{align*}
    We have now a representation formula for $\nabla E$ by the expression above and \eqref{eq:nablaE1}:
    \begin{equation*}
    \nabla E(x) =  -\left(Sx\cdot \nabla\right)\nabla\psi(x) - S^T\nabla\psi(x)  + \int_{\R^3} \frac{\xi\cdot (S-J)A^{-1}\xi}{\xi\cdot A^{-1}\xi}\hat{\nabla\psi}(\xi)\,e^{2\pi i \xi\cdot x}\,d\xi.
    \end{equation*}
    To obtain \eqref{eq:LSG}, we just need to observe that
    \begin{align*}
    \xi\cdot (S-J)A^{-1}\xi 
    = 2(\xi\cdot A^{-1}J\xi).
    \end{align*}
\end{proof}

\begin{rmk}[Degeneracy of the pseudo-differential operator]
    The flow of LSG is interesting for the presence of the pseudo-differential operator $\mathscr F^{-1}\circ \mathscr M_m\circ \mathscr F$, so we show that the action of this operator is not trivial, i.e. we show that  the pseudo-differential operator does not vanish for all choices of the matrix $A$.
    In fact, $\mathscr F^{-1}\circ \mathscr M_{m}\circ \mathscr F[\nabla \psi]= 0$ for a non-trivial $\nabla\psi$ if and only if $m=0$ a.e. on  $\R^3$.
    Therefore, we have that the operator $\mathscr F^{-1}\circ \mathscr M_{m}\circ \mathscr F$ is the zero operator if the matrix $A$ belongs to the set of matrices
    \begin{equation}\label{eq:badA}
        \mathscr B:= \left\{ A\in\text{Sym}^+_3(\R) : x\cdot (A^{-1}J)x = 0\ \forall x\in\R^3 \right\}.
    \end{equation}
    We denote the complement of $\mathscr B$ by $\mathscr G:= \text{Sym}^+_3(\R) \setminus \mathscr B$.
    An example of a matrix in the set $\mathscr B$ is $A=\beta I$, that describes a non-trivial elliptic flow when $\beta\in (0,1)\cup(1,\infty)$: 
    \begin{equation*}
        \overline P_{\beta}(x)= \frac{\beta}2 |x|^2, \qquad \overline u_{\beta}(x) = \frac{\beta -1}{\beta}Jx.
    \end{equation*}
    In particular, we observe that $\mathscr A_0\subsetneq \mathscr B$, as any matrix $A\in\mathscr A_0$ admits a null steady full velocity $\bar u = 0$, which gives a trivial multiplier $m=0$. On the other hand, the matrix $A={\beta} I$ with $\beta\in(0,1)\cup(1,\infty)$ corresponds to a trivial multiplier $m$ but non-trivial steady full velocity $\bar u_{\beta}$. This implies that matrices in the set $\mathscr B$ can generate elliptic and/or hyperbolic flows.
    We will discuss the definitions and the difference between elliptic and hyperbolic flows in Section \ref{sec:LSG_stability}.
\end{rmk}


\section{Existence Theory for LSG}\label{sec:LSG_existence}

In this section, we present the existence results for LSG in $L^2(\R^3,\R^3)$, then in $\mathcal S$, and finally in $\mathcal S'$, where we define strong solutions by duality. In fact, as we mentioned before, the operator $\mathscr L$ defined in \eqref{eq:LSG_L} and the operator $\mathscr L'=\mathscr K$ defined in \eqref{eq:LSG_K} are the same up the a sign and the transposition of the matrix $S$, therefore any result for LSG \eqref{eq:LSG_abs} can be proved for the abstract Cauchy problem \eqref{eq:ALSG}.

Although the existence of solutions in $L^2(\R^3,\R^3)$ can be proved by using the theory of strongly continuous semigroups on Banach spaces, we give a direct proof using the explicit representation formula of the solution. Formally, given the initial datum $\phi_0$, the solution $\phi$ can be written as $\phi(\cdot,t)=G(t)\phi_0$, where the operator $G(t)$ is defined in \eqref{eq:LSG_rep_form}. We provide here a formal derivation of the formula.

First of all, we derive LSG in Lagrangian coordinates, which we introduced in \eqref{eq:LSG_Lagrangian}. Given $\phi$ a solution of LSG \eqref{eq:LSG_abs}, we define the function
\begin{align}\label{eq:LSG_Lag_Phi}
    \Phi(x,t):= e^{tS^T}\phi(e^{tS}x,t), \qquad (x,t)\in\R^3\times\R,
\end{align}
and observe that 
\begin{align*}
    \frac{\partial}{\partial t}\Phi(x,t) 
    &= \frac{\partial}{\partial t} \left(e^{tS^T}\phi(e^{tS}x,t)\right)\\
    &= S^T e^{tS^T}\phi(e^{tS}x,t) + e^{tS^T}\frac{\partial}{\partial t}\phi(e^{tS}x,t) + e^{tS^T} \left(\frac{\partial}{\partial t}e^{tS}x\cdot \nabla\right) \phi(e^{tS}x,t)\\
    &=e^{tS^T}\left[ \frac{\partial}{\partial t}\phi(e^{tS}x,t) + S^T \phi(e^{tS}x,t) + \big( Se^{tS}x\cdot \nabla \big) \phi(e^{tS}x,t)  \right]\\
    &=e^{tS^T} \left[\mathscr F^{-1}\circ \mathscr M_m \circ\mathscr F[\phi(\cdot,t)](e^{tS}x)\right]\\
    &=e^{tS^T} \int_{\R^3}\int_{\R^3}m(\xi)\phi(y,t)e^{-2\pi i y\cdot \xi}e^{2\pi i \xi\cdot e^{tS}x}\,dy\,d\xi\\
    &=\int_{\R^3}\int_{\R^3} m(\xi) \Phi(e^{-tS}y,t)e^{-2\pi i y\cdot \xi}e^{2\pi i \xi\cdot e^{tS}x}\,dy\,d\xi\\
    &=\int_{\R^3}\int_{\R^3} \tilde m(t;\xi) \Phi(y,t) e^{-2\pi i y\cdot \xi}e^{2\pi i \xi\cdot x}\,dy\,d\xi\\
    &=\mathscr F^{-1}\circ \mathscr M_{\tilde m(t;\cdot)}\circ \mathscr F[\Phi(\cdot,t)](x),
\end{align*}
which corresponds to the equation \eqref{eq:LSG_Lagrangian}. We can now formally take the Fourier transform to write explicitly the solution $\Phi$: we have that $\Phi$ solves LSG in Lagrangian coordinates \eqref{eq:LSG_Lagrangian} if and only if $\widehat{\Phi(\cdot,t)}$ solves the abstract Cauchy problem
\begin{equation*}
    \left\{\begin{array}{l}
    \frac{\partial}{\partial t} \hat{\Phi(\cdot,t)} = \tilde m(t)\hat{\Phi(\cdot,t)},\vspace{2mm}\\
    \hat{\Phi(\cdot,0)}= \hat{\phi_0},
    \end{array}
    \right.
\end{equation*}
If we multiply both the sides by $\tilde M(t;\cdot)^{-1}=e^{-\int_0^t \tilde m(r;\cdot)\,dr}$, we see that
\begin{align*}
    \frac{\partial}{\partial t} \left(\tilde M(t;\cdot)^{-1}\hat{\Phi(\cdot,t)}\right) = 0, 
\end{align*}
and therefore $\Phi(\cdot,t) = \mathscr F^{-1}[\tilde M(t;\cdot)\hat{\phi_0}]$ is the solution of LSG in Lagrangian coordinates \eqref{eq:LSG_Lagrangian}. By the definition of the function $\Phi$ in \eqref{eq:LSG_Lag_Phi}, we have that 
\begin{align*}
    \phi(x,t) &= e^{-tS^T}\Phi(e^{-tS}x,t)\\
    &= \int_{\R^3} \tilde M(t;\xi) e^{-tS^T}\hat{\phi_0}(\xi) e^{2\pi i \xi\cdot e^{-tS}x}\,d\xi\\
    &=G(t)\phi_0(x).
\end{align*}

The derivation above is formal, and we now prove that $G(t)\phi_0$ is indeed a solution of LSG in the appropriate topologies, starting with $L^2(\R^3,\R^3)$. 

\subsection{Existence of Solutions of LSG in \texorpdfstring{$L^2(\R^3,\R^3)$}{L2}}

\begin{thm}\label{thm:LSG_existence_L2}
    Given $\phi_0\in L^2(\R^3,\R^3)$, the function $\phi(\cdot,t):=G(t)\phi_0$, where the operator $G(t)$ is defined in \eqref{eq:LSG_rep_form}, is a weak solution in $L^{2}(\R^3,\R^3)$ of \eqref{eq:LSG_abs}. 
    Moreover, if $\phi_0$ belongs to the domain $D(\mathscr L)\subset L^2(\R^3,\R^3)$ of the operator $\mathscr L$, then $\phi$ is a strong solution of LSG in $L^{2}(\R^3,\R^3)$.
    Furthermore, if $\phi_0$ is conservative, namely $\phi_0=\nabla f_0\in L^2(\R^3,\R^3)$, then $G(t)\phi_0$ is conservative for all times $t\in\R$.
\end{thm}

\begin{proof}
    The existence and uniqueness of weak/strong solutions in $L^2(\R^3,\R^3)$ can be proved in several ways: for instance, it is easy to show that $\mathscr L$ is a bounded perturbation of the infinitesimal generator of a strongly continuous group on $L^2(\R^3,\R^3)$, and the existence (and uniqueness) of solutions follows from classical semigroup theory on Banach space. In this work, we make use instead of the representation formula, by showing that $G(t)\phi_0$ defines a weak/strong solution of LSG in $L^2(\R^3,\R^3)$, depending on the choice of the initial data.
    We start by showing that the operators $T_1(t)$, $T_2(t)$ and $T_3(t)$ satisfy certain properties, namely
    \begin{align*}
        \|T_j(t)\phi\|_{L^2}\lesssim_t \|\phi\|_{L^2}, \qquad \|T_j(t)\phi-T_j(r)\phi\|_{L^2} \to 0 \quad \text{as}\quad r\to t, 
    \end{align*}
    for $j=1,2,3$. 
    
    For the family of operators $\{T_1(t)\}_{t\in\R}$, we notice that $\|T_1(t)\phi\|_{L^2(\R^3,\R^3)}=\|\phi\|_{L^2(\R^3,\R^3)}$ for any $\phi \in L^2(\R^3,\R^3)$, because $x\mapsto e^{-tS}x$ is a volume-preserving diffeomorphism on $\R^3$. We can also prove that
    \begin{align*}
        \lim_{r\to t} \|T_1(t)\phi-T_1(r)\phi\|_{L^2}\leq\lim_{h\to 0} \|T_1(h)\phi-\phi\|_{L^2} =0,
    \end{align*}
    for all $\phi\in L^2(\R^3,\R^3)$. In fact, by density of $\mathcal C^{\infty}_0(\R^3,\R^3)$ in $L^2(\R^3,\R^3)$, for any $\phi\in L^2(\R^3,\R^3)$ and for any $\eps>0$ there exists a function $f\in \mathcal C^{\infty}_0(\R^3,\R^3)$ such that $\|\phi-f\|_{L^2}<\eps$. Therefore, we have that 
    \begin{align*}
        \|T_1(h)\phi - \phi\|_{L^2} \leq \|T_1(h)(\phi-f)\|_{L^2} + \|f-\phi\|_{L^2}+\|T_1(h)f-f\|_{L^2}< 2\eps + \|T_1(h)f-f\|_{L^2}.
    \end{align*}
    Now, since $f$ is smooth, we can apply the mean value theorem and find a $\dl>0$ such that $\|T_1(h)f-f\|_{L^2}<\eps$ if $|h|<\dl$, proving the convergence of the required limit.
    Regarding the operators $\{T_2(t)\}_{t\in\R}$, we clearly have that $\|T_2(t)\phi\|_{L^2(\R^3,\R^3)} \leq e^{|t||S|}\|\phi\|_{L^2(\R^3,\R^3)}$ and, as $r\to t$,
    \begin{align*}
        \|T_2(t)\phi-T_2(r)\phi\|_{L^2(\R^3,\R^3)} = \|(e^{-tS^T}-e^{-rS^T})\phi\|_{L^2(\R^3,\R^3)}\to0.
    \end{align*}
    Finally, by Plancherel theorem, we have that $\| T_3(t)\phi\|_{L^2(\R^3,\R^3)}\leq  e^{|t|\|m\|_{L^{\infty}(\R^3)}}\|\phi\|_{L^2(\R^3,\R^3)}$. Since the symbol $\tilde M$ is smooth in $t$, we also have that, as $r\to t$,
    \begin{align*}
        \|T_3(t)\phi-T_3(r)\phi\|_{L^2(\R^3,\R^3)}\leq \|\tilde M(t;\cdot)-\tilde M(r;\cdot)\|_{L^{\infty}(\R^3)}\|\phi\|_{L^2(\R^3,\R^3)} \to 0.
    \end{align*}
    
    By using the estimates above, we have that the operator $G(t)$ is defined from $L^2(\R^3,\R^3)$ to $L^2(\R^3,\R^3)$ for any $t\in\R$. Moreover, the function $t\mapsto G(t)\phi_0$ is sequentially continuous: let $t_n\to t$ in $\R$, then
    \begin{align*}
        \|G(t_n)\phi_0 -& G(t)\phi_0\|_{L^2} 
        \leq \|(T_1(t_n)-T_1(t)) T_2(t_n) T_3(t_n)\phi_0\|_{L^2}\\
        &+\|T_1(t)(T_2(t_n)-T_2(t)) T_3(t_n)\phi_0\|_{L^2} +\|T_1(t)T_2(t)( T_3(t_n)- T_3(t))\phi_0\|_{L^2}\to 0.
    \end{align*}
    We need to prove that $\phi$ is a weak solution of LSG in $L^2(\R^3,\R^3)$, so we look at the following integral: for any test function $\psi\in\mathcal C^{\infty}_C(\R^3\times [0,\infty),\R^3)$
    \begin{align*}
        \int_{\R}\int_{\R^3} &G(t)\phi_0(x)\cdot \left(\partial_t \psi(x,t) + \mathscr K\psi(x,t)\right)\,dx\,dt=\int_{\R^3} \phi_0(x)\cdot \int_{\R}G(t)'\left(\partial_t\psi(x,t)+\mathscr K\psi(x,t)\right)\,dt\,dx,        
    \end{align*}
    where $G(t)'$ is just the adjoint operator to $G(t)$ in $L^2(\R^3,\R^3)$ for any $t\in\R$, and one can prove that it is given by 
    \begin{align}\label{eq:LSG_G'}
        G(t)'[\phi](x) = e^{-tS}\int_{\R^3} \bar M(t;\xi)\hat {\phi} (\xi) e^{2\pi i \xi\cdot e^{tS}x}\,d\xi.
    \end{align}
    From a manipulation of the pseudo-differential operator $G(t)'$, it follows that $G(t)'\left(\partial_t\psi(x,t)+\mathscr K\psi(x,t)\right) = \partial_t\left(G(t)'\psi(x,t)\right)$, which shows that $G(t)\phi_0$ is a weak solution of LSG in $L^2(\R^3,\R^3)$, as in Definition \ref{de:LSG_solutions}.
    
    We now consider the case of $\phi_0\in D(\mathscr L)$ and observe that $\phi$ is a strong solution of LSG. Through a calculation, one can prove that $\partial_t \phi(\cdot,t) = \mathscr L \phi(\cdot,t)$ for any $t\in\R$, and we can prove continuity of $t\mapsto \partial_t\phi(\cdot,t)$ in a similar fashion as before.

    Finally, we consider a conservative initial datum $\phi_0=\nabla f_0\in L^2(\R^3,\R^3)$ and the corresponding weak solution $\phi(\cdot,t)=G(t)\nabla f_0$. Through a standard computation, one can show that 
    \begin{align*}
        \phi(x,t) = \nabla \int_{\R^3} \tilde M(t;\xi) \hat f_0(\xi) e^{2\pi i \xi\cdot e^{-tS}x}\,d\xi,
    \end{align*}
    so the weak solution is a gradient for all times $t\in\R$.


\end{proof}

\begin{rmk}
    The weak/strong solution constructed in the theorem above is actually unique: this is a straightforward consequence of the fact that $\mathscr L$ is the infinitesimal generator of the strongly continuous group $\{G(t)\}_{t\in\R}$. As we mentioned above, we do not provide a proof of this fact in this work, but we invite the reader to see \cite[Ch.~4]{Pazy12} for the relevant theory.
\end{rmk}

\subsection{Existence of Solutions of LSG in \texorpdfstring{$\mathcal S(\R^3,\R^3)$}{S}}

We now focus on the well-posedness of LSG in the Fr\'echet space $\mathcal S$: as we mentioned in the introduction, we endow $\mathcal S$ with the weak topology generated by the seminorms \eqref{eq:S_seminorms}. For the well-posedness in $\mathcal S$, one way would be to directly apply the theory of locally equicontinuous semigroups on locally convex topological spaces by showing that $L_1$, $L_2$ and $L_3$ are infinitesimal generators of locally equicontinuous $C_0$-groups, and then apply Trotter formula (see \cite[Theorem 20]{AK:Trotter}). The problem with this way of proceeding is that the exponent in \eqref{eq:bound_T1} below depends on the choice of the seminorm $\|\cdot\|_{\al,\be}$. Therefore, we use instead the representation formula \eqref{eq:LSG_rep_form} and show that this gives a strong solution in $\mathcal S$ as by Definition \ref{de:LSG_solutions}. In order to do so, we will need estimates for the families of operators $\{T_1(t)\}_{t\in\R}$, $\{T_2(t)\}_{t\in\R}$ and $\{ T_3(t)\}_{t\in\R}$, in a similar fashion as what we did for the $L^2$-theory.

\begin{thm}\label{thm:LSG_existence_S}
    Given $\phi_0\in \mathcal S$, there exists a unique strong solution $\phi$ in $\mathcal S$ of LSG, as in Definition \ref{de:LSG_solutions}. Moreover, the solution is given by $\phi(t)=G(t)[\phi_0]$, as defined in \eqref{eq:LSG_rep_form}.
\end{thm}
 
Before proving the theorem above, we need the following auxiliary result.

\begin{lem}\label{lem:LSG_commutativity}
    For any $t\in\R$, the two operators $\mathscr L$ and $G(t)$ commute on $\mathcal S$. 
\end{lem}

\begin{proof}
    By Definition \ref{de:LSG_semigroups}, 
    we have that $\mathscr L=L_1+L_2+L_3$, whereas $G(t)=T_1(t)T_2(t) T_3(t)$. One can easily notice that $L_2$ and $T_2(t)$ are just given by multiplication by $S^T$ and $e^{tS^T}$ respectively, and that they commute with all the other operators involved. Therefore, it suffices to prove that the operator
    \begin{equation}\label{eq:LSG_commut_operators}
        T_1(t)T_3(t)L_1+ T_1(t)T_3(t)L_3 - L_1 T_1(t)T_3(t) - L_3T_1(t)T_3(t)
    \end{equation}
    is the zero operator on $\mathcal S$.
    We start by looking at the first term, which is morally given by the representation formula for the solution of \eqref{eq:LSG_abs} applied to the transport term in LSG \eqref{eq:LSG_abs}. For any $\phi\in\mathcal S$ and $x\in\R^3$ we have the following:
    \begin{align*}
        T_1(t)T_3(t)L_1 [\phi](x)
        &= \int_{\R^3}\int_{\R^3} \tilde M(t;\xi) S^T_{jk}y_k \partial_{y_j}\phi(y)
        e^{-2\pi i y\cdot \xi} e^{2\pi i \xi\cdot e^{-tS}x}\,dy\,d\xi\\
        &=\int_{\R^3}\int_{\R^3} \tilde M(t;\xi)\left(2\pi i S^T y\cdot \xi - \Tr S\right)\phi(y) e^{-2\pi i y\cdot \xi}e^{2\pi i \xi\cdot e^{-tS}x}\,dy\,d\xi.
    \end{align*}
    The second term in \eqref{eq:LSG_commut_operators} is given by the composition of two pseudo-differential operators: by standard calculations, one can show that for any $\phi\in\mathcal S$ the following holds:
    \begin{align*}
        T_1(t)T_3(t) L_3[\phi](x) 
        &=\int_{(\R^3)^4} \tilde M(t;\xi) m(\eta)\phi(z) e^{2\pi i z\cdot \eta} e^{-2\pi i \eta\cdot y} e^{-2\pi i y\cdot \xi} e^{2\pi i \xi\cdot e^{-tS}x}\,dz\,d\eta\,dy\,d\xi\\
        &=\int_{\R^3}\int_{\R^3} \tilde M(t;\xi) \phi(z) m(\xi) e^{-2\pi i z\cdot \xi}e^{2\pi i\xi\cdot e^{-tS}x}\,dz\,d\xi.
    \end{align*}
     For the third term in \eqref{eq:LSG_commut_operators}, we use the Einstein summation convention and consider the transport term applied to the representation formula for the solution of \eqref{eq:ALSG}, up to a multiplication by the matrix $e^{-tS}$:  
    \begin{align*}
        L_1T_1(t)T_3(t)[\phi](x) 
        &=(S^T x\cdot \nabla) \int_{\R^3}\int_{\R^3} \tilde M(t;\xi) \phi(y)e^{-2\pi i y\cdot \xi} e^{2\pi i \xi\cdot e^{-tS}x}\,dy\,d\xi\\
        &=-\int_{\R^3}\int_{\R^3} \phi(y) \partial_{\xi_k} \left( \tilde M(t;\xi)e^{-2\pi i y\cdot \xi} S_{kj}\xi_j\right) e^{2\pi i \xi\cdot e^{-tS}x}\,dy\,d\xi\\
        &=\int_{\R^3}\int_{\R^3} \tilde M(t;\xi) \big( 2\pi i S^Ty\cdot \xi- \Tr S\big)\phi(y) e^{-2\pi i y\cdot \xi}e^{2\pi i \xi\cdot e^{-tS}x}\,dy\,d\xi\\
        &\qquad -\int_{\R^3} \big(S\xi\cdot \nabla \tilde M(t;\xi)\big)\hat{\phi}(\xi)e^{2\pi i \xi\cdot e^{-tS}x}\,d\xi.
    \end{align*}
     In the last term in \eqref{eq:LSG_commut_operators}, we have again the composition of two pseudo-differential operators: by using their definition, one can easily see that
    \begin{align*}
        L_3 T_1(t)T_3(t)[\phi](x)
        &=\int_{(\R^3)^4} m(\xi)\tilde M(t;\eta) \phi(z)e^{-2\pi i z\cdot \eta} e^{2\pi i \eta\cdot e^{-tS}y} e^{2\pi i y\cdot \xi}e^{-2\pi i \xi\cdot x}\, dz\,d\eta\,dy\,d\xi\\
        &=\int_{\R^3}\int_{\R^3} \tilde M(t;\eta)\phi(z)\tilde m(t;\eta) e^{-2\pi i z\cdot\eta}e^{2\pi i \eta\cdot e^{-tS}x}\,dz\,d\eta.
    \end{align*}
     Therefore, the operator in \eqref{eq:LSG_commut_operators}  is simply given by the following pseudo-differential operator: 
    \begin{align*}
       \bigg( T_1(t)T_3(t)L_1 +& T_1(t)T_3(t)L_3 - L_1 T_1(t)T_3(t) - L_3T_1(t)T_3(t)\bigg)[\phi](x)\\
        &=\int_{\R^3} \left(\tilde M(t;\xi)m(\xi) + S\xi \cdot \tilde M(t;\xi) -\tilde M(t;\xi)\tilde m(t;\xi)\right)\hat{\phi}(\xi) e^{2\pi i \xi\cdot e^{-tS}x}\,d\xi.
    \end{align*}
    The symbol is in fact null for any $\xi\in\R^3\setminus\{0\}$ and $t\in\R$, because $S\xi \cdot \tilde M(t;\xi)=\tilde M(t;\xi)\tilde m(t;\xi)- \tilde M(t;\xi)m(\xi)$, as one can show by considering the following: 
    \begin{align*}
        S\xi\cdot \nabla\tilde M(t;\xi) 
        &=\tilde M(t;\xi) \int_0^t S\xi\cdot  e^{-rS} \nabla m(e^{-rS^T}\xi)\,dr\\
        &=\tilde M(t;\xi)\int_0^t \partial_t \left(m(e^{-rS^T}\xi)\right)\,dr\\
        &=\tilde M(t;\xi) \tilde m(t;\xi) - \tilde M(t;\xi) m(\xi).
    \end{align*}
    This ends the proof of the lemma.
\end{proof}

\begin{proof}[Proof of Theorem \ref{thm:LSG_existence_S}]
    By a calculation, one can easily prove that the function $\phi(x,t):= G(t)[\phi_0](x)$ is classically differentiable in $x$ and $t$ and it solves the PDE for any $t\in\R$ and $x\in\R^3$. We establish that the solution is indeed a strong solution with respect to the weak topology in the Schwartz space: to prove this, we show that the function $t\mapsto G(t)\phi_0$ is continuously differentiable from $\R$ to $\mathcal S$ by showing that its time derivative is sequentially continuous. We first need to show that the operators in the families $\{T_1(t)\}_{t\in\R}$, $\{T_2(t)\}_{t\in\R}$ and $\{T_3(t)\}_{t\in\R}$ are bounded and $T_i(r)\phi\to T(t)\phi$ as $r\to t$ for any $\phi\in\mathcal S$, $i\in\{1,2,3\}$. In fact, one could show that the families $\{T_1(t)\}_{t\in\R}$, $\{T_2(t)\}_{t\in\R}$ and $\{T_3(t)\}_{t\in\R}$ are locally equicontinuous $C_0$-groups on the locally convex topological space $\mathcal S$, but this fact is not needed in our proof, and we just make use of the estimates.
    
    By induction on $|\beta|\in\N$, one can show that for any $\al,\be\in\N^3$ there exist constants $C(\al)>0$ and $M(\al,\be)\in\N$ and a finite family of seminorms $\left\{\|\cdot\|_{\al^{(1)},\be^{(1)}}\dots\|\cdot\|_{\al^{(M(\al,\be))},\be^{(M(\al,\be))}}\right\}$ such that 
        \begin{equation}\label{eq:bound_T1} 
            \|T_1(t)\phi\|_{\al,\be} \leq e^{|t|\,|S|_{\infty}(|\al|+|\be|)} C(\al) \sum_{i=1}^{M(\al,\be)} \|\phi\|_{\al^{(i)},\be^{(i)}},
        \end{equation}
    for any $\phi\in\mathcal S$ and for any $t\in\R$.
    Again, by induction on $|\beta|$, one can also show that $\|T_1(t)\phi-\phi\|_{\al,\be}\to 0$ as $t\to 0$, for any $\phi\in\mathcal S$.
    The operators $T_2(t)$ are clearly bounded, as 
    \begin{align}\label{eq:bound_T2}
        \|T_2(t)\phi\|_{\al,\be} \leq e^{|t|\,|S|_{\infty}} \|\phi\|_{\al,\be},
    \end{align}
    for any $\phi\in\mathcal S$ and $t\in\R$, and the map $t\mapsto T_2(t)\phi$ is clearly continuous for all $\phi\in\mathcal S$.

    In order to show that the family of operators $\{T_3(t)\}_{t\in\R}$ also satisfies the wanted properties, the main tool is given by the following estimate for $\mathscr F$ and $\mathscr F^{-1}$: for any $\al,\be\in\N^3$ we have that 
    \begin{align}\label{eq:bound_FT_in_S}
        \|\mathscr F[\phi]\|_{\al,\be} \leq\sum_{\substack{k\in\N^3\\k\leq\al,\be}} {C}_k^{\al,\be}\left[ \|\phi\|_{\be-k,\al-k} + \sum_{i,j=1}^3 \|\phi\|_{\be-k+2e_i+2e_j,\al-k}\right],
    \end{align}
    where the constant $C_k^{\al,\be}$ is defined as 
    \begin{equation*}
        C_k^{\al,\be} := \sqrt 2\pi^2 (2\pi)^{|\be-\al|}k!\binom{\al}{k}\binom{\be}{k}.
    \end{equation*}
    In particular, the following bound 
    \begin{equation}\label{eq:bound_T3}
        \| T_3(t)\phi\|_{\al,\be} \leq  C e^{|t|\,\|m\|_{L^{\infty}(\R^3)}} \sum_{i=1}^{N} \|{\phi}\|_{\al^{(i)},\be^{(i)}},
    \end{equation}
    holds for any $\phi\in\mathcal S$ and $t\in\R$. The continuity of the map $t\mapsto T_3(t)\phi$ follows from the fact that the symbol $\tilde M$ is smooth in $t$.
    
    The estimates \eqref{eq:bound_T1}, \eqref{eq:bound_T2} and \eqref{eq:bound_T3} and the fact that the map $t\mapsto T_i(t)\phi_0$ is continuous from $\R$ to $\mathcal S$, for $\phi_0\in\mathcal S$ and $i\in\{1,2,3\}$, allow us to prove that the map $t\mapsto \frac d{dt}G(t)\phi$ is continuous for any $\phi_0\in\mathcal S$: as $r\to t$ in $\R$, we have that    
    \begin{align*}
        \big\|\frac d{dt}G(r)\phi_0-&\frac d{dt}G(t)\phi_0\big\|_{\al,\be}
        =\|\mathscr L G(r)\phi_0-\mathscr L G(t)\phi_0\|_{\al,\be}
        =\|G(r)[\mathscr L \phi_0]- G(t)[\mathscr L \phi_0]\|_{\al,\be}\\
        &\leq\|\big(T_1(r)-T_1(t)\big)T_2(r)T_3(r)[\mathscr L \phi_0]\|_{\al,\be} +\|T_1(t)\big(T_2(r)-T_2(t)\big)T_3(r)[\mathscr L \phi_0]\|_{\al,\be}\\
        &\qquad +\|T_1(t)T_2(t)\big(T_3(r)-T_3(t)\big)[\mathscr L \phi_0]\|_{\al,\be}
    \end{align*}
    converges to $0$ as $r\to t$, for any $\al,\be\in\N^3$. This proves that the function $t\mapsto \frac d{dt}G(t)\phi_0$ is continuous and therefore \eqref{eq:LSG_rep_form} is a strong solution of LSG in $\mathcal S$.
    
    The last part that we need to prove to conclude the proof is the uniqueness of such solution in the space $ \mathcal C^1(\R,\mathcal S)$. Since the operator $\mathscr L$ is linear, it suffices to prove that the only solution of \eqref{eq:LSG_abs} with $\phi_0=0$ is the trivial solution $\phi=0$. Denote by $\phi^{(0)}$ a strong solution of LSG in $\mathcal S$ corresponding to $\phi(0)=0$, then $\phi^{(0)}(t)$ is square-integrable for any $t\in\R$: by an energy estimate, it is easily shown that $\|\phi^{(0)}(t)\|_{L^2(\R^3,\R^3)}=0$ for any $t\in\R$. Therefore $\phi^{(0)}=0$, proving uniqueness of strong solutions in $\mathcal S$. 
\end{proof}

As we mentioned above, the adjoint LSG problem \eqref{eq:ALSG} is mathematically equivalent to LSG \eqref{eq:LSG_abs}, therefore the proof of Theorem \ref{thm:LSG_existence_S} can be adapted to show the following result. 

\begin{cor}\label{cor:ALSG_existence_S}
    Given $\psi_0\in \mathcal S$, there exists a unique strong solution $\psi\in \mathcal C^1(\R,\mathcal S)$ to the abstract Cauchy problem \eqref{eq:ALSG}. Moreover, the solution is given by $\psi(t)=F(t)\psi_0$, where the operator $F(t)$ is defined as
    \begin{align}\label{eq:ALSG_rep_form}
        F(t)\psi(x)=e^{-tS}\int_{\R^3} \bar M(t;\xi)\hat{\psi}(\xi)e^{2\pi i \xi\cdot e^{tS}x}\,d\xi,
    \end{align}
    for any $\psi\in\mathcal S$.
\end{cor}
The reader can compare the expressions \eqref{eq:ALSG_rep_form} and \eqref{eq:LSG_G'} and observe that $F(t)=G(t)'$, namely $F(t)$ is the adjoint operator of $G(t)$. 


\subsection{Existence of Solutions of LSG  in \texorpdfstring{$\mathcal S'(\R^3,\R^3)$}{S'}}

The existence of strong solutions of \eqref{eq:LSG_abs} in $\mathcal S$ allows us to construct strong solutions to LSG in $\mathcal S'$. 

\begin{proof}[Proof of Theorem \ref{thm:LSG_existence_S'}]
    For any $t\in\R$, we consider the operator $F(t)$ defined in Corollary \ref{cor:ALSG_existence_S}, and we prove that $t\mapsto F(t)'T_0$ is a strong solution of LSG in $\mathcal S'$, in the sense of the Definition \ref{de:LSG_solutions}. With $F(t)'$ we denote the operator on $\mathcal S'$ defined by duality as 
    \begin{align*}
        \ip{F(t)'\eta}{\psi}=\ip{\eta}{F(t)\psi}, \qquad \forall \psi\in\mathcal S,
    \end{align*}
    for a distribution $\eta\in\mathcal S'$. In fact, the operator $F(t)'$ can be seen as an ``extension'' of the operator $G(t)$ to the space of tempered distributions. In order to prove that $F(t)'\eta_0$ is a strong solution of LSG, we apply Lemma \ref{lem:LSG_commutativity}: in fact, for any $\eta\in\mathcal S'$ and $t\in\R$, the following sequence of identities 
    \begin{align*}
        \ip{\frac d{dt} F(t)'\eta}{\psi} &=\ip{\eta}{\frac d{dt}F(t)\psi}=\ip{\eta}{\mathscr K F(t)\psi}=\ip{\eta}{F(t)\mathscr K\psi}=\ip{\mathscr L F(t)' \eta}{\psi}
    \end{align*}
    holds for any test function $\psi\in\mathcal S$. Continuous  differentiability of the function $t\mapsto F(t)'\eta_0$ from $\R$ to $\mathcal S'$ is a consequence of the continuous differentiability of $t\mapsto F(t)\psi_0$ from $\R$ to $\mathcal S$ for any $\psi_0\in\mathcal S$, as proved in Theorem \ref{thm:LSG_existence_S}.
\end{proof}

As we mentioned, $F(t)'$ is formally an extension of $G(t)$ to $\mathcal S'$. This is clarified by the following corollary, that connects the $L^2$ theory of Theorem \ref{thm:LSG_existence_L2} and the $\mathcal S'$ theory of Theorem \ref{thm:LSG_existence_S'}.

\begin{cor}
    Let $\eta_0$ be a regular distribution corresponding to the function $\phi_0\in L^2(\R^3,\R^3)$. Then the solution $\eta$ constructed in the proof of Theorem \ref{thm:LSG_existence_S'} is regular for all times $t\geq 0$, i.e. $\eta(t)=\eta_{\phi(t)}$. The function $\phi$ is given by the representation formula $\phi(t)=G(t)\phi_0$, as defined in \eqref{eq:LSG_rep_form} and it is the unique weak solution of LSG in $L^2$. If $\phi_0\in D(\mathscr L)\subset L^2(\R^3,\R^3)$, then $\phi$ is the unique strong solution of LSG in $L^2$. 
\end{cor}

\begin{proof}  
    The definition of the solution $\eta$ that was constructed in the proof of Theorem \ref{thm:LSG_existence_S'} and a standard calculation allow us to write for any test function $\psi\in\mathcal S$ and for any $t\in\R$
    \begin{align*}
        \ip{\eta(t)}{\psi} = \ip{F(t)'\eta_0}{\psi} = \ip{\eta_0}{F(t)\psi} = \int_{\R^3} \phi_0\cdot F(t)\psi = \int_{\R^3} G(t)\phi_0\cdot\psi= \ip{\eta_{G(t)\phi_0}}{\psi}.
    \end{align*}
    By Theorem \ref{thm:LSG_existence_L2}, we have that $G(t)\phi_0$ is a weak (strong) solution of LSG in $L^2(\R^3,\R^3)$ if $\phi_0\in L^2(\R^3,\R^3)$ ($\phi_0\in D(\mathscr L)$), and if $\phi_0$ is conservative so is $G(t)\phi_0$ for all times $t\in\R$ by Theorem \ref{thm:LSG_existence_L2}.    

\end{proof}

\section{Stability}\label{sec:LSG_stability}

As we now have an existence theory of solutions of LSG, we can study the long-time behaviour of its solutions and comment on stability of the steady solutions introduced in \eqref{eq:steadysol}. A comprehensive theory of stability of such solutions is not the purpose of this work, and we instead consider the particular case of plane-wave perturbations. We begin by characterising the specific forms of the amplitude $a$ and frequency $k$ that give rise to solutions of LSG of the form \eqref{eq:LSG_planewavesol}, writing their form explicitly in terms of their initial values $a(0)$ and $k(0)$. We then look at their stability, presenting a proof of Theorem \ref{thm:LSG_stability}.

\subsection{Plane-wave Solutions}

We now consider a specific type of solutions to LSG: we are interested in solutions $\phi$ of the form \eqref{eq:LSG_planewavesol}. When we say that $\phi$ is a solution, we mean that the corresponding regular distribution $\eta_{\phi}$ is a strong solution of LSG in $\mathcal S'$ with initial datum $\eta_{\phi(\cdot,0)}$. We first characterise the functions $a$ and $k$ that generate a strong solution $\eta_{\phi}$.

\begin{pr}
    The function $\phi$ of the form in \eqref{eq:LSG_planewavesol} corresponds to a strong solution $\eta_{\phi}$ in $\mathcal S'$ if and only if the functions $a$ and $k$ are of the form
    \begin{equation}\label{eq:LSG_ak}
        a(t)=a_0\tilde M(t;k_0), \qquad k(t)=e^{-tS^T}k_0,
    \end{equation}
    for $a_0\in\R\setminus\{0\}$ and $k_0\in\R^3\setminus\{0\}$, and $\tilde M$ defined in \eqref{eq:LSG_def_symbols}.
\end{pr}

\begin{proof}
 If $\eta_{\phi}$ is the regular distribution corresponding to a plane-wave solution $\phi$, then $\eta_{\phi}$ is a strong solution of LSG if and only if for any test function $\psi\in\mathcal S$ and $t\in\R$ we have that
 \begin{align*}
    0&=\ip{\frac d{dt} \eta_{\phi(\cdot,t)} }{\psi} - \ip{\eta_{\phi(\cdot,t)}}{\mathscr K\psi}=\int_{\R^3} \left[\partial_t \phi(x,t) - \mathscr L \phi(x,t)\right]\cdot \psi(x)\,dx,
 \end{align*}
 which holds if and only if $\partial_t \phi -\mathscr L \phi=0$. 
 From a direct computation and separating the real and imaginary parts of the quantity $\partial_t \phi -\mathscr L \phi$, one proves that $\phi$ can generate a strong regular solution of LSG in $\mathcal S'$ if and only if $a$ and $k$ solve the following initial value problems:
\begin{equation*}
    \left\{\begin{array}{l}
            a'(t) = m(k(t))a(t),\\
            a(0)=a_0,
           \end{array}
    \right. 
    \qquad 
    \left\{\begin{array}{l}
            k'(t) = -S^T k(t),\\
            k(0)= k_0.
           \end{array}
    \right.
\end{equation*}
The functions in \eqref{eq:LSG_ak} are the unique solutions to the Cauchy problems above.
\end{proof}
Using the definition of $a$ and $k$ in \eqref{eq:LSG_ak}, we notice that we can write the solution $\phi$ as 
\begin{equation*}
    \phi(x,t) = a_0 \tilde M(t;k_0) e^{-tS^t}k_0 e^{2\pi i k_0\cdot e^{-tSx}},
\end{equation*}
which formally coincides with $G(t)[a_0 k_0e^{2\pi i k_0\cdot x}]$, the representation formula that we found for solutions in $L^2$ and in $\mathcal S$.

\subsection{Stability of the Plane-wave Solutions}

We proceed now to prove Theorem \ref{thm:LSG_stability}, distinguishing between elliptic and hyperbolic flows, as we defined them in \eqref{eq:flowsets}.

\begin{proof}[Proof of Theorem \ref{thm:LSG_stability}]
 We use different arguments for the two types of flows, so we look at one case at the time. 
 
 {\em I. Hyperbolic flows.} 
    For a matrix $A\in\mathscr A_+\cap \mathscr G$, the spectrum of $S$ is $\{0,\ld,-\ld\}$, with $\ld>0$. We denote with $k_0^{\pm}$ an eigenvector of $S^T$ corresponding to the eigenvalue $\pm \ld$, and observe that $\tilde m(t;k_0^{\pm})= \pm 2\ld$: in fact,
    \begin{align*}
        \pm \ld 
        &= \frac{S^T k_0^{\pm}\cdot A^{-1}k_0^{\pm}}{k_0^{\pm}\cdot A^{-1}k_0^{\pm}}\\
        &= \frac{k_0^{\pm}\cdot (A^{-1}J - A^{-1}J A^{-1})k_0^{\pm}}{k_0^{\pm}\cdot A^{-1}k_0^{\pm}}\\
        &=\frac{e^{\mp t\ld}k_0^{\pm}\cdot A^{-1}J e^{\mp t\ld}k_0^{\pm}}{e^{\mp t\ld}k_0^{\pm}\cdot A^{-1}e^{\mp t\ld}k_0^{\pm}} \\
        &=\frac{e^{-tS^T}k_0^{\pm}\cdot A^{-1}J e^{-tS^T}k_0^{\pm}}{e^{-tS^T}k_0^{\pm}\cdot A^{-1}e^{-tS^T}k_0^{\pm}} \\
        &=\frac 12 m(e^{-tS^T}k_0^{\pm}) = \frac 12 \tilde m(t;k_0^{\pm}).
    \end{align*}
    This allows us to write $\tilde M(t;k_0^{\pm}) = e^{\pm 2t\ld}$ and to estimate the $L^{\infty}$-norm of the solution $\phi$ at all times:
    \begin{align*}
        \|\phi(\cdot,t)\|_{L^{\infty}(\R^3,\R^3)} 
        =\sup_{x\in\R^3} |\phi(x,t)|
        = |a_0| e^{\pm 2 t\ld} e^{\mp t\ld}|k_0^{\pm}|  
        = |a_0| |k_0^{\pm}|e^{\pm t\ld}.
    \end{align*}
    Hence, if we choose the solution $\phi$ corresponding to an initial datum with $k_0$ eigenvector corresponding to the positive eigenvalue, the solution grows exponentially in time. Likewise, if $k_0$ is chosen in the eigenspace corresponding to the negative eigenvalue, the solution decays exponentially in time. 
    
    If instead $A\in\mathscr A_+\cap \mathscr B$, then $a$ is constant as $m=0$, and the solution $\phi$ is given by 
    \begin{align*}
        \phi(x,t) = a_0 e^{-tS^T} k_0.
    \end{align*}
    If $k_0^{\pm}$ is eigenvector of $S^T$ with eigenvalue $\pm\ld$, then $|\phi(x,t)|= |a_0|e^{\mp t\ld}|k_0|$. Therefore the solution with initial datum $\phi_0(x)=a_0k_0e^{2\pi i k_0\cdot x}$, with $k_0$ eigenvector of $S^T$ corresponding to the negative eigenvalue, grows exponentially in $L^{\infty}$-norm.

 {\em II. Elliptic flows.}
    For a general $A\in\mathscr A_-\cap \mathscr G$, the spectrum of the matrix $S$ is $\{0,i\ld,-i\ld\}$, with $\ld>0$. The trajectories $k(t)$ are bounded and periodic with period $\tau =\frac{2\pi}{\ld}$, hence the long-term behaviour of $\phi$ depends on $\tilde M(t;k_0)$, as 
    \begin{align*}
        |\phi(x,t)| = |k(t)| |a_0| \tilde M(t;k_0) = |k(t)||a_0|e^{\int_0^t m(k(r))\,dr}. 
    \end{align*}
    Since $k(t)$ is $\tau$-periodic, the integral of $m(k(t))$ can be partitioned as below 
    \begin{align*}
        \int_0^t m(k(r))\,dr = N\int_0^{\tau} m(k(r))\,dr + \int_0 ^{t-N\tau} m(k(r))\,dr, 
    \end{align*}
    where $N=\left\lfloor \frac t{\tau}\right\rfloor$. The behaviour of the solution depends on the sign of the integral of $m(k(t))$ over the period $[0,\tau]$:
    we show that 
    \begin{align}\label{eq:LSG_ell_mean}
        \int_0^{\tau} m(k(r))\,dr = 0,
    \end{align}
    therefore the integral of $m(k(r))$ over $[0,t]$ oscillates around $0$, proving that $\phi$ is bounded for any initial data $a_0$ and $k_0$:
    \begin{align*}
        \|\phi\|_{L^{\infty}(\R^3\times\R,\R^3)}\leq \max_{[0,\tau]}|k| |a_0| e^{\tau\|m\|_{L^{\infty}}}.
    \end{align*}
    
    We need to prove that \eqref{eq:LSG_ell_mean} holds for any initial data $a_0$ and $k_0$. First of all, we observe that we can focus on the numerator: if $k_0\neq 0$, then $k(t)\neq 0$ for all $t$, and there exists a constant $C>1$ such that for all $t\in\R$
    \begin{align*}
        0< \frac 1C\leq k(t)\cdot A^{-1}k(t)\leq C<\infty. 
    \end{align*}
    We use the notation $p(x):=x\cdot A^{-1}J x$ for the numerator of $m$, and we show that there exists a function $f:\R\to\R$ such that $p(k(t))=f'(t)$. Since $\sigma(S^T)=\{0,i\ld, -i\ld\}$, there exists an invertible matrix $V\in\R^{3\times 3}$ such that $S^T= -\mu VJV^{-1}$, where $J$ is defined in \eqref{eq:def_J}, and therefore
    \begin{align*}
        p(k(t))&= k_0\cdot e^{-tS}A^{-1}Je^{-tS^T}k_0 = k_0\cdot (V e^{\mu t J} V^{-1})^T A^{-1}J Ve^{\mu t J}V^{-1} k_0 \\
        &= k_0 \cdot V^{-T} e^{-\mu tJ} V^TA^-1 J V e^{\mu tJ} V^{-1} k_0,
    \end{align*}
    where 
    \begin{align*}
        e^{\mu t J} = \begin{pmatrix}
                       \cos(\mu t) & -\sin(\mu t) & 0\\
                       \sin(\mu t) &  \cos(\mu t) & 0\\
                       0           &              & 1
                      \end{pmatrix}.
    \end{align*}
    Hence, $p(k(t)) = c_1 \cos^2(\mu t) + c_2 \cos(\mu t)\sin(\mu t) + c_3\cos (\mu t) + c_4 \sin(\mu t) + c_5$, with $c_1,\dots,c_5\in\R$ constants depending on $k_0$ and $A$. This expression admits a primitive $f(t)$ given by
    \begin{align*}
        f(t) = c_1 \frac{2\mu t + \sin(2\mu t)}{2\mu} - c_2 \frac{\cos^2(\mu t)}{2\mu} + c_3\frac{\sin (\mu t)}{\mu} - c_4 \frac{\cos(\mu t)}{\mu} + c_5 t + c_6,
    \end{align*}
    and \eqref{eq:LSG_ell_mean} holds if and only if $f(\tau) = f(0)$, namely if and only if $c_1 +2c_5=0$. This works for any choice of the matrix $A\in\mathscr A_-\cap\mathscr G$.
    
    Finally, if $A\in\mathscr A_-\cap \mathscr B$, then $a$ is constant and $|\phi(x,t)| = |a_0| |k(t)|$ is bounded, concluding the proof. 

\end{proof}

\section{Comparison with the Quasi-geostrophic Approximation}\label{sec:LSG_LQG}

We now present a brief comparison between the stability of the family of steady solutions \eqref{eq:steadysol} which are solutions of both SG and the quasi-geostrophic equations (QG). The two models are obtained as asymptotic limits of the $3$-D Euler equations for small Rossby number, and the theory presented in this paper can be proved for QG following exactly the same steps. In fact, rather than writing QG in terms of both full and geostrophic velocities  (as presented, for instance, in \cite{Vallis}), one can show that the equations can be written in terms of the full velocity $\mathscr U[\nabla\phi]$ and the gradient $\nabla{{\phi}}$ as follows:
\begin{equation}\label{eq:QG2}
    \frac{\partial}{\partial t} \nabla {\phi} + (J\nabla{\phi}\cdot \nabla)\nabla{\phi} + B\mathscr U[\nabla {\phi}] = J\nabla{\phi},
\end{equation}
where $B=\text{diag}(1,1,N^2)$, $N$ is the Brunt--V\"ais\"al\"a frequency and the operator $\mathscr U: \nabla{\phi}\mapsto u$ is the solution operator associated to the div-curl system
\begin{equation*}
\left\{\begin{array}{l}
    \nabla\times(B u) = \nabla\times J\nabla{\phi} - \nabla\times D^2{\phi} J\nabla {\phi},\vspace{2mm}\\
    \nabla\cdot u=0,
    \end{array}\right.
\end{equation*}
endowed with decay conditions on the velocity field. 
The streamfunction $\phi$ is related to the geopotential $P$ in the semi-geostrophic equation \eqref{eq:SG} through the identity 
\begin{equation*}
    P(x,t) = {\phi}(x,t) + \frac 12 x\cdot Bx.
\end{equation*}
Therefore, the steady solutions for QG corresponding to \eqref{eq:steadysol} are given by
\begin{align}\label{eq:QG_steadysol}
    \bar{\phi}(x,t) = \frac 12 x\cdot \tilde A x,
\end{align}
with  $\tilde A = A-B$ and $A$ a positive definite symmetric matrix. 
We mention that $\phi$ is not the physical pressure $p$, which is instead given by $p(x,t) = \phi(x,t) + \frac{N^2}2 x_3^2$, for any $x=(x_1,x_2,x_3)\in\R^3$ and $t\in\R$.
 
 \subsection{Linearisation of QG}
 
 In a similar fashion as we did in Section \ref{sec:LSG_deriv}, the quasi-geostrophic equation \eqref{eq:QG2} can be linearised at a steady solution $\overline{\phi}$ in \eqref{eq:QG_steadysol} and written as the following linear problem in the unknown vector field $\psi$
 \begin{equation}\label{eq:LQG}
  \left\{\begin{array}{l}
          \frac{\partial}{\partial t}\psi + (\overline u_g\cdot \nabla)\psi(x) + M^T\psi(x) = \mathscr F^{-1}\mathscr M_{m_{QG}}\mathscr F[\psi](x),\vspace{2mm}\\
          \psi(\cdot,0)= \psi_0,
         \end{array}
\right.
 \end{equation}
 where the steady geostrophic velocity $\overline u_g$ is given by $\overline u_g(x)=Mx$, with the matrix $M$ defined as $M:= J(A-B)$. The symbol $m_{QG}$ has a similar structure as the symbol $m$ defined in \eqref{eq:def_m} for LSG:
 \begin{align}\label{eq:def_m_QG}
  m_{QG}(x) := 2\frac{x\cdot MB^{-1} x}{x\cdot B^{-1}x}.
 \end{align}
 We refer to the initial value problem \eqref{eq:LQG} as the \emph{linearised quasi-geostrophic equation}, or LQG.
 
 By studying the spectrum of the matrix $M$, the reader can show that we can again identify three possible behaviours: we have that $\sigma(M)=\{0,\sqrt{\mu_{QG}}, -\sqrt{\mu_{QG}}\}$, with $\mu_{QG}=b^2-(a-1)(d-1)$, where $a,b,c,d,e,f$ are the coefficients of the matrix $A$ as in \eqref{eq:LSG_def_A}, and therefore the set of positive definite symmetric matrices can be partitioned into the three disjoint sets:
\begin{align}\label{eq:QG_flowsets}
    &\mathscr A_+^{QG} :=\left\{A\in\text{Sym}_+(3):\mu(A)>0\right\},\\
    &\mathscr A_-^{QG} :=\left\{A\in\text{Sym}_+(3):\mu(A)<0\right\},\\
    &\mathscr A_0^{QG} :=\left\{A\in\text{Sym}_+(3):\mu(A)=0\right\}.
\end{align}
We call hyperbolic and elliptic, respectively, the flows associated to the first and second sets, as we do in the SG theory. As we observed for SG, the set $\mathscr A_0^{QG}$ corresponds to $\overline u^g=0$ and it is not interesting, therefore we focus on the two open sets $\mathscr A_+^{QG}$ and $\mathscr A_-^{QG}$.

It is a straight-forward exercise to prove that the results that we prove for LSG in Section \ref{sec:LSG_existence} and Section \ref{sec:LSG_stability} hold also for LQG. In particular, we define the operator
\begin{align}\label{eq:LQG_L}
 \mathscr L_{QG} \psi := -(\overline u_g\cdot \nabla)\psi(x) - M^T\psi(x) + \mathscr F^{-1}\mathscr M_{m_{QG}}\mathscr F[\psi](x),
\end{align}
and we can derive the representation formula for the solution associated to the initial data $\psi_0$ as 
\begin{align*}
    \psi(x,t) = \int_{\R^3} e^{-tM^T}\widehat{\psi_0}(\xi)\widetilde M_{QG}(t;\xi) e^{2\pi i \xi\cdot e^{-tM}x}\,d\xi,
\end{align*}
where $\widetilde M_{QG}$ is defined depending on $m_{QG}$ in \eqref{eq:def_m_QG}:
\begin{align*}
    \widetilde m_{QG}(t;\xi)=m_{QG}(e^{-tM^T}\xi),\qquad \widetilde M_{QG}(t;\xi) = e^{\int_0^t \widetilde m_{QG}(r;\xi)\,dr}.
\end{align*}

\begin{thm}\label{thm:LQG_existence_S'}
     Given an initial tempered distribution $\eta_0\in\mathcal S'(\R^3,\R^3)$, there exists a strong solution in $\mathcal S'(\R^3,\R^3)$ of \eqref{eq:LQG}, i.e. there exists a $\eta\in\mathcal C^1(\R;\mathcal S'(\R^3,\R^3))$ such that for all $t\in\R$ we have
     \begin{align*}
      \ip{\frac{d}{dt}\eta(t)}{\psi} = \ip{\mathscr L_{QG}\eta(t)}{\psi}, \qquad \forall \psi\in\mathcal S(\R^3,\R^3),
     \end{align*}
     and $\eta(0)=\eta_0$ in $\mathcal S'(\R^3,\R^3)$.
\end{thm}

We consider plane-wave solutions of LQG \eqref{eq:LQG} of the form \eqref{eq:LSG_planewavesol}, and consider the corresponding concepts of stability to plane-wave perturbations, as we did in Definition \ref{de:LSG_stability}. 

\begin{thm}\label{thm:LQG_stability}
    The steady solution $\bar \phi$ of QG \eqref{eq:QG2} defined in \eqref{eq:QG_steadysol} associated to a matrix $A$ is stable to plane-wave perturbations if it corresponds to an elliptic flow, i.e. $A\in\mathscr A_-^{QG}$ as in \eqref{eq:QG_flowsets}, and it is unstable to plane-wave perturbations if it corresponds to an hyperbolic flow, i.e. $A\in\mathscr A_+^{QG}$.
\end{thm}

\subsection{Stability in SG and QG}

We showed that the matrix-parameterised family of common steady solutions for SG and QG defined in \eqref{eq:steadysol} and \eqref{eq:QG_steadysol} exhibit different behaviours depending on the choice of the positive definite symmetric matrix $A$, and in either approximations we partitioned the set of such matrices into two open subsets (as we already mentioned that the sets $\mathscr A_0$ and $\mathscr A_0^{QG}$ are not interesting). In fact, the two partitions do not coincide, and therefore they allow four different possibilities: referring to the sets defined in \eqref{eq:flowsets} and \eqref{eq:QG_flowsets}, we have that either
\begin{itemize}
 \item [(i)] $A\in \mathscr A_{+}\cap \mathscr A_{+}^{QG}$, or
 \item [(ii)] $A\in \mathscr A_{+}\cap \mathscr A_{-}^{QG}$, or
 \item [(iii)] $A\in \mathscr A_{-}\cap \mathscr A_{+}^{QG}$, or
 \item [(iv)] $A\in \mathscr A_{-}\cap \mathscr A_{-}^{QG}$.
\end{itemize}

Each of the interections above is non-empty and contains at least one matrix that generates a non-trivial flow. The following matrices are examples of elements in the intersections above, respectively:
\begin{align*}
 A_1=\begin{pmatrix}
      \frac 12 & -1 & -1 \\
      -1 & 4 & 1 \\
      -1 & 1 & 3
     \end{pmatrix}, \quad
 A_2= \begin{pmatrix}
       2 & 0 & -1 \\
       0 & 2 & 0 \\
       -1 & 0 & \frac 34
      \end{pmatrix},\quad
 A_3=\begin{pmatrix}
      \frac 12 & -1 & -1\\
      -1 & 3 & 3 \\
      -1 & 3 & \frac 72
     \end{pmatrix}, \quad
 A_4=\begin{pmatrix}
      \frac 12 & 0 & -1\\
      0 & \frac 12 & 0 \\
      -1 & 0 & 3
     \end{pmatrix}.
\end{align*}
The fact that the intersections (ii) and (iii) are non-empty suggests that, although SG and QG are heuristically derived to model the same phenomena, they in fact exhibit different behaviours. 

\section{Final Remarks}

We introduced the linearisation of the semi-geostrophic equation \eqref{eq:SG} for globally-defined strictly convex quadratic steady solutions and presented existence results on $L^2(\R^3,\R^3)$, $\mathcal S(\R^3,\R^3)$ and $\mathcal S'(\R^3,\R^3)$. We discussed the stability of such steady solutions, by introducing the concept of stability to plane-wave perturbations, following the approach in \cite{CC86:NS_planewaves}. We also notice that the same techniques can be applied to the well-known quasi-geostrophic equations, and drew a comparison between the long-term behaviour of the same solutions in the two models, LSG \eqref{eq:LSG} and LQG \eqref{eq:LQG}. 

The authors believe that this work is a first step in analysing the dynamical stability of steady solutions to SG \eqref{eq:SG}. In particular, it would be interesting to consider more general steady solutions: in Section \ref{sec:LSG_deriv}, the reader can observe that the derivation makes use of the fact that the Hessian $D^2\overline P=A$ of the steady geopotential is constant, and lifting this assumption would lead to a different and less explicit form for the operator $\mathscr L$ defined in \eqref{eq:LQG_L}. Another issue is to construct non-trivial steady solutions of SG \eqref{eq:SG} both on the whole space $\R^3$ and on a domain with boundary: quadratic steady solutions can be constructed if the domain is a ball, but more interesting domains require more complicated solutions, which make it harder to keep  calculations as explicit as they are in this paper. 

Another point on which the authors would like to draw attention is the concept of stability: our definition of stability to plane-wave perturbations arises from the idea of perturbing the steady solution and studying the long-term behaviour of the perturbation in order to infer something about the behaviour of solutions which are ``close'' to the steady state. Firstly, we want to remind the reader that a rigorous stability analysis requires a well-posedness theory for the nonlinear equation \eqref{eq:SG} that is not available at the moment. Secondly, the concept of stability can be extended to more general perturbations. Again, this can require a more abstract approach, whereas the plane-wave solutions considered in this paper allow for a more explicit investigation of stability.

\section*{Acknowledgments}

S. Lisai was supported by The Maxwell Institute Graduate School in Analysis and its
Applications (MIGSAA) funded by the UK Engineering and Physical
Sciences Research Council (grant EP/L016508/01), the Scottish Funding Council, Heriot-Watt
University and the University of Edinburgh.
M. Wilkinson was supported by the EPSRC Standard Grant EP/P011543/1.
  

\addcontentsline{toc}{section}{References}
\bibliographystyle{alpha}
\bibliography{biblio}

\begin{thebibliography}{ACDPF14}

\bibitem[ACDPF12]{ACDF12:periodic}
L.~Ambrosio, M.~Colombo, G.~De~Philippis, and A.~Figalli.
\newblock Existence of {E}ulerian solutions to the semigeostrophic equations in
  physical space: the 2-dimensional periodic case.
\newblock {\em Communications in Partial Differential Equations},
  37(12):2209--2227, 2012.

\bibitem[ACDPF14]{ACDF12:convex}
L.~Ambrosio, M.~Colombo, G.~De~Philippis, and A.~Figalli.
\newblock A global existence result for the semigeostrophic equations in three
  dimensional convex domains.
\newblock {\em Discrete and Continuous Dynamical Systems - A}, 34:1251, 2014.

\bibitem[AK02]{AK:Trotter}
A.~Albanese and F.~K\"{u}hnemund.
\newblock Trotter-{K}ato approximation theorems for locally equicontinuous
  semigroups.
\newblock {\em Rivista di Matematica della Universita di Parma}, 7(1):19--53,
  2002.

\bibitem[BB98]{BB98}
J.-D. Benamou and Y.~Brenier.
\newblock Weak existence for the semigeostrophic equations formulated as a
  coupled {Monge-Amp\`{e}re}/transport problem.
\newblock {\em SIAM J. Appl. Math.}, 58(5):1450--1461, October 1998.

\bibitem[CC86]{CC86:NS_planewaves}
A.~D.~D. Craik and W.~O. Criminale.
\newblock Evolution of wavelike disturbances in shear flows : a class of exact
  solutions of the navier-stokes equations.
\newblock {\em Proceedings of the Royal Society of London. A. Mathematical and
  Physical Sciences}, 406(1830):13--26, July 1986.

\bibitem[CKPW18]{CKPW19}
M.~J.~P. Cullen, T.~Kuna, B.~Pelloni, and M.~Wilkinson.
\newblock Cullen's stability principle and weak solutions of the free-surface
  semi-geostrophic equations.
\newblock {\em arXiv preprint arXiv:1811.03926}, 2018.

\bibitem[CS87]{CS87}
M.~J.~P. Cullen and G.~J. Shutts.
\newblock Parcel stability and its relation to semigeostrophic theory.
\newblock {\em Journal of Atmospheric Sciences}, 44:1318--1330, May 1987.

\bibitem[Cul06]{Cullen:book}
M.~J.~P. Cullen.
\newblock {\em {A Mathematical Theory of Large-Scale Atmosphere/Ocean Flow}}.
\newblock Imperial College Press, 2006.

\bibitem[Eli48]{Eliassen48}
A.~Eliassen.
\newblock {The quasi-static equations of motion with pressure as independent
  variable}.
\newblock {\em Geofis. Publ.}, 17(3):5--44, 1948.

\bibitem[Hos75]{Hoskins75}
B.~J. Hoskins.
\newblock The geostrophic momentum approximation and the semi-geostrophic
  equations.
\newblock {\em Journal of Atmospheric Sciences}, 32:233--242, February 1975.

\bibitem[LW19]{LW19}
S.~Lisai and M.~Wilkinson.
\newblock Smooth solutions of the surface semi-geostrophic equations.
\newblock {\em Calculus of Variations and Partial Differential Equations},
  59(1), 2019.

\bibitem[Paz12]{Pazy12}
A.~Pazy.
\newblock {\em Semigroups of linear operators and applications to partial
  differential equations}, volume~44.
\newblock Springer Science \& Business Media, 2012.

\bibitem[Ren98]{Ren98}
S.~Ren.
\newblock Linear stability of the three-dimensional semigeostrophic model in
  geometric coordinates.
\newblock {\em Journal of the atmospheric sciences}, 55(22):3392--3402, 1998.

\bibitem[Val17]{Vallis}
G.~K. Vallis.
\newblock Geostrophic theory.
\newblock In {\em Atmospheric and Oceanic Fluid Dynamics}, pages 171--212.
  Cambridge University Press, 2017.

\bibitem[Yos95]{Yosida}
K.~Yosida.
\newblock {\em Functional Analysis}.
\newblock Springer Berlin Heidelberg, 1995.

\end{thebibliography}

\end{document}